\documentclass[]{amsart}
\usepackage[OT2,T1]{fontenc}
\usepackage{amsmath}
\usepackage{amssymb}
\usepackage{yfonts}
\usepackage{mathtools}
\usepackage{txfonts}
\usepackage[mathscr]{euscript}
\usepackage{calrsfs}
\usepackage{hyperref}
\usepackage{epsfig}
\newtheorem{theorem}{Theorem}
\newtheorem{corollary}{Corollary}
\newtheorem{definition}{Definition}

\newtheorem{lemma}{Lemma}

\newtheorem{remark}{Remark}
\newcommand{\eps}{\varepsilon}
\DeclareMathAlphabet{\mathpzc}{OT1}{pzc}{m}{it}

\DeclareMathOperator{\Graph}{Gr}

\DeclareMathOperator{\fix}{Fix}
\DeclareMathOperator{\F}{F}

\DeclareMathOperator{\es}{K}

\DeclareMathOperator{\A}{A}
\DeclareMathOperator{\ka}{k}
\DeclareMathOperator{\co}{co}

\DeclareMathOperator{\dist}{dist}

\DeclareMathOperator{\spn}{span}

\newcommand{\w}{\tilde}
\newcommand{\map}{\multimap}
\newcommand{\<}{\leqslant}

\newcommand{\h}{{\mathscr H}}

\newcommand{\n}{{n\geqslant 1}}
\newcommand{\K}{{k\geqslant 1}}
\newcommand{\e}{{\mathscr E}}
\newcommand{\z}[1]{(#1)_{n=1}^\infty}
\newcommand{\x}[1]{\{#1\}_{n=1}^\infty}
\newcommand{\R}[1]{\mathbb{R}^{#1}}

\newcommand{\f}{\left}
\newcommand{\g}{\right}
\newcommand{\res}[2]{#1\:\rule[-1.5mm]{0.45pt}{4mm}\,\rule[-1mm]{0mm}{4mm}_{#2}}

\begin{document}
\title[Solvability of inclusions of Hammerstein type]{Solvability of inclusions of Hammerstein type}
\author{Rados\l aw Pietkun}
\address{Toru\'n, Poland}
\email{rpietkun@pm.me}

\begin{abstract}
We establish a universal rule for solving operator inclusions of Hammerstein type in Lebesgue--Bochner spaces with the aid of some recently proven continuation theorem of Leray--Schauder type for the class of so-called admissible multimaps. Examples illustrating the legitimacy of this approach include the initial value problem for perturbation of $m$-accretive mutivalued differential equation, the nonlocal Cauchy problem for semilinear differential inclusion, abstract integral inclusion of Fredholm and Volterra type and the two-point boundary value problem for nonlinear evolution inclusion.
\end{abstract}
\maketitle

\section{Introduction}
This paper aims to formulate quite natural and easily verifiable hypotheses, ensuring solvability of the following inclusion of Hammerstein type
\begin{equation}\label{hammer}
u\in (K\circ N_F)(u)
\end{equation}
in the space $L^p(I,E)$ of Bochner $p$-integrable functions. In inclusion \eqref{hammer}, $N_F$ is the Nemytski\v{\i} operator associated to a multifunction $F\colon I\times E\map E$, while $K$ is an external set-valued operator of a certain type (defined later).\par Consideration of such operator inclusion accompany, of course, attempts to grasp the integro-{diffe\-ren\-tial} multivalued problems from a unifying topological point of view. These attempts have been made repeatedly (see for instance \cite{cor,couch1,couch2,precup,drici,jank,lupu,zecca}). Our efforts follow in the footsteps of authors of \cite{precup} by taking into account the situation, where the operator $K$ is not only nonlinear but possibly multivalued and does not necessarily have a quasi-integral form. The main result (Theorem \ref{solvability}.) regarding the existence of solutions to inclusion \eqref{hammer} poses an application example of a fixed point approach. Its proof is based on the principle of Leray--Schauder type \cite[Theorem 3.2]{regan}, which was extended (Theorem \ref{fixed}.) to the case of admissible multimaps (in the sense of G\'orniewicz, \cite[Definition 40.1]{gorn}). Just as in \cite{precup} the superposition $K\circ N_F$ may not be a condensing map and our assumptions about $K$ and $F$ are formulated so that the M\"onch type compactness condition could have been satisfied.\par In order to apply Eilenberg--Montgomery type fixed point argument directly to the superposition $K\circ N_F$ we need to know that this map is pseudo-acyclic. However, the Nemytski\v{\i} operator $N_F\colon L^p\map L^q$ is by no means acyclic. Therefore, the authors of \cite{precup} rely on the assumption (SG) that operator $H:=K\circ N_F$ has acyclic values. This is very uncomfortable hypothesis from practical point of view. In general, if $K$ is nonlinear, then the composite map $H$ may not have convex values. Unfortunately, even if $K$ and $F$ has convex values the map $H$ may still have values with ``awful'' geometry, since the class of acyclic mappings is not closed with respect to the composition law. It turns out that it is enough to take into account a relatively weak assumption regarding convexity or decomposability of fibers of the operator $K$ in addition to the acyclicity of its values, to ensure the fulfillment of condition (SG).\par The applicability of our abstract existence result is richly illustrated by numerous examples of differential and integral inclusions, which may be interpreted as a fixed point problem given by \eqref{hammer}. These examples include cases where the operator $K$ is a univalent mild solution operator of the $m$-accretive quasi-autonomous problem or the mild solution operator of the semilinear inhomogeneous two-point boundary value problem. There were also presented examples in which the map $K$ is simply linear. Such as those, in which it has the form of Volterra or Hammerstein integral operator. And finally, there is also the case considered, when the map $K$ constitutes a multivalued strongly upper semicontinuous maximal monotone operator.

\section{Preliminaries}
Let $(E,|\cdot|)$ be a Banach space, $E^*$ its normed dual and $\sigma(E,E^*)$ its weak topology. If $X$ is a subset of a Banach space $E$, by $(X,w)$ we denote the topological space $X$ furnished with the relative weak topology of $E$. The symbol $(X,|\cdot|)$ stands for the topological space $X$ endowed with the restriction of the norm-topology of $E$ to $X$.\par The normed space of bounded linear endomorphisms of $E$ is denoted by $\mathcal{L}(E)$. Given $T\in\mathcal{L}(E)$, $||T||_{{\mathcal L}}$ is the norm of $T$. For any $\eps>0$ and $A\subset E$, $B_E(A,\eps)$ ($D_E(A,\eps)$) stands for an open (closed) $\eps$-neighbourhood of the set $A$. The (weak) closure and the closed convex envelope of $A$ will be denoted by ($\overline{A}^w$) $\overline{A}$ and $\overline{\co}\,A$, respectively. If $x\in E$ we put $\dist(x,A):=\inf\{|x-y|\colon y\in A\}$. Besides, for two nonempty closed bounded subsets $A$, $B$ of $E$ the symbol $h(A,B)$ stands for the Hausdorff distance from $A$ to $B$, i.e. $h(A,B):=\max\{\sup\{\dist(x,B)\colon x\in A\},\sup\{\dist(y,A)\colon y\in B\}\}$.\par We denote by $(C(I,E),||\cdot||)$ the Banach space of all continuous maps $I\to E$ equipped with the maximum norm. Let $1\<p\<\infty$. By $(L^p([a,b],E),||\cdot||_p)$ we mean the Banach space of all Bochner $p$-integrable maps $f\colon[a,b]\to E$ i.e., $f\in L^p([a,b],E)$ iff map f is strongly measurable and \[||f||_p=\left(\int_a^b|f(t)|^p\,{\rm d}t\right)^{\frac{1}{p}}<\infty\] if $p<\infty$ and respectively 
\[||f||_\infty=\underset{t\in[a,b]}{{\rm ess}\,{\rm sup}}\,|f(t)|<\infty\] provided $p=\infty$.
Recall that strong measurability is equivalent to the usual measurability in case $E$ is separable. A subset $D\subset L^p([a,b],E)$ is called decomposable if for every $u, w\in D$ and every Lebesgue measurable $A\subset [a,b]$ we have $u\cdot{\bf 1}_A+w\cdot{\bf 1}_{A^c}\in D$. 
\par Given metric space X, a set-valued map $F\colon X\map E$ assigns to any $x \in X$ a nonempty subset $F(x)\subset E$. $F$ is (weakly) upper semicontinuous, if the small inverse image $F^{-1}(A)=\{x\in X\colon F(x)\subset A\}$ is open in $X$ whenever $A$ is (weakly) open in $E$. A map $F\colon X\map E$ is lower semicontinuous, if the inverse image $F^{-1}(A)$ is closed in $X$ for any closed $A\subset E$. We say that $F\colon X\map E$ is upper hemicontinuous if for each $p\in E^*$, the function $\sigma(p,F(\cdot))\colon X\to\R{}\cup\{+\infty\}$ is upper semicontinuous (as an extended real function), where $\sigma(p,F(x))=\sup_{y\in F(x)}\langle p,y\rangle$. We have the following characterization (\cite[Proposition 2(b)]{bothe}): a map $F\colon X\map E$ with convex values is weakly upper semicontinues and has weakly compact values iff given a sequence $(x_n,y_n)$ in the graph $\Graph(F)$ with $x_n\xrightarrow[n\to\infty]{X}x$, there is a subsequence $y_{k_n}\xrightharpoonup[n\to\infty]{E}y\in F(x)$ ($\rightharpoonup$ denotes the weak convergence). A multifunction $F\colon X\map E$ is compact if its range $F(X)$ is relatively compact in $E$. It is quasicompact if its restriction to any compact subset $A\subset X$ is compact. The set of all fixed points of the map $F\colon E\map E$ is denoted by $\fix(F)$. \par Let $H^\ast(\cdot)$ denote the Alexander-Spanier cohomology functor with coefficients in the field of rational numbers ${\mathbb Q}$ (see \cite{spanier}). We say that a topological space $X$ is acyclic if the reduced cohomology $\w{H}^q(X)$ is $0$ for any $q\geqslant 0$.
\begin{definition}[\protect{\cite[Definition 40.1]{gorn}}]\label{admis}
Let $\mathbb{F}$ be a Fr\'echet space and $\Gamma$ a Hausdorff topological space. A surjective continuous map $p\colon\Gamma\to\mathbb{F}$ is called a Vietoris map if $p$ is closed and, for every $x\in\mathbb{F}$, the set $p^{-1}(x)$ is compact acyclic. A set-valued map $F\colon\mathbb{F}\map\mathbb{F}$ is admissible, if there is a Hausdorff topological space $\Gamma$ and two continuous functions $p\colon\Gamma\to\mathbb{F}$ and $q\colon\Gamma\to\mathbb{F}$ from which $p$ is a Vietoris map such that $F(x)=q(p^{-1}(x))$ for every $x\in\mathbb{F}$.
\end{definition}
An admissible set-valued map is an upper semicontinuous one with compact connected values. If the aforesaid Vietoris map $p\colon\Gamma\to\mathbb{F}$ is not only proper but open at the same time, then the admissible set-valued map $F\colon\mathbb{F}\map\mathbb{F}$ is also continuous. An upper semicontinuous map $F\colon\mathbb{F}\map\mathbb{F}$ is called acyclic if it has compact acyclic values. Evidently, every acyclic map is admissible. Moreover, the composition of admissible maps is admissible (\cite[Theorem 40.6]{gorn}).\par A real function $\gamma$ defined on the family ${\mathcal B}(E)$ of bounded subsets of $E$ is called a measure of non-compactness (MNC) if $\gamma(\Omega)=\gamma(\overline{\co}\,\Omega)$ for any bounded subset $\Omega$ of $E$. The following example of MNC is of particular importance: given $E_0\subset E$ and $\Omega\in{\mathcal B}(E_0)$, 
\[\beta_{E_0}(\Omega):=\inf\f\{\eps>0:\text{ there are finitely many points }x_1,\dots,x_n\in E_0\text{ with }\Omega\subset\bigcup_{i=1}^nB_E(x_i,\eps)\g\}\]
is the Hausdorff MNC relative to the subspace $E_0$. Recall that this measure is regular i.e., $\beta_{E_0}(\Omega)=0$ iff $\Omega$ is relatively compact in $E_0$; monotone i.e., if $\Omega_1\subset\Omega_2$ then $\beta_{E_0}(\Omega_1)\<\beta_{E_0}(\Omega_2)$ and invariant with respect to union with compact sets i.e., $\beta_{E_0}(A\cup\Omega)=\beta_{E_0}(\Omega)$ for any relatively compact $A\subset E_0$.
\par We recall the reader following results on account of their practical importance. The first is a weak compactness criterion in $L^p(\Omega,E)$, which originates from \cite{ulger}.
\begin{theorem}[\protect{\cite[Corollary 9.]{ulger}}]\label{ulger}
Let $(\Omega,\Sigma,\mu)$ be a finite measure space with $\mu$ being a nonatomic measure on $\Sigma$. Let $A$ be a uniformly $p$-integrable subset of $L^p(\Omega,E)$ with $p\in[1,\infty)$. Assume that for a.a. $\omega\in\Omega$, the set $\{f(\omega)\colon f\in A\}$ is relatively weakly compact in $E$. Then $A$ is relatively weakly compact.
\end{theorem}
\begin{remark}
{\em The genuine formulation of this result assumes the boundedness of the set $A$. However, the fact that $\mu$ is nonatomic means that uniform integrability of $A$ entails its boundedness.}
\end{remark}

The next property is commonly known as the Convergence Theorem for upper hemicontinuous maps with convex values (the mentioned below version is proved in \cite{pietkun3}).
\begin{theorem}\label{convergence}
Let $(\Omega,\Sigma,\mu)$ be a $\sigma$-finite measure space, $E$ a Banach and $F\colon E\map E$ a closed convex valued upper hemicontinuous multimap. Assume that functions $f_n,f\colon\Omega\to E$ and $g_n,g\colon\Omega\to E$ are such that
\[\begin{cases}
g_n(x)\xrightarrow[n\to\infty]{E}g(x)\;\text{a.e. on }\Omega,\\
f_n\xrightharpoonup[n\to\infty]{L^1(\Omega,E)}f,\\
f_n(x)\in\overline{\co}\,B(F(B(g_n(x),\eps_n)),\eps_n)\;\text{a.e. on }\Omega,
\end{cases}\]
where $\eps_n\xrightarrow[n\to\infty]{}0^+$. Then $f(x)\in F(g(x))$ a.e. on $\Omega$.
\end{theorem}

The third result is a Lefschetz-type fixed point theorem for admissible multimaps.
\begin{theorem}[\protect{\cite[Theorem 7.4]{gorn2}}]\label{Lefschetz}
Let $X$ be an absolute extensor for the class of compact metrizable spaces and $F\colon X\map X$ be an admissible map such that $F(X)$ is contained in a compact metrizable subset of $X$. Then $F$ has a fixed point.
\end{theorem}

\section{Fixed point approach to inclusions of Hammerstein type}

The subsequent results constitute a generalization of the continuation principle \cite[Theorem 3.2]{regan} to the case of admissible multimaps.

\begin{theorem}\label{fixed}
Let $\mathbb{F}$ be a Fr\'echet space and $X$ a nonempty closed convex subset of $\mathbb{F}$. Assume that $U$ is relatively open in $X$ and its closure is a retract of $X$. Assume further that $F\colon\overline{U}\map X$ is an admissible set-valued map and for some $x_0\in U$ the following two conditions are satisfied:
\begin{equation}\label{monch}
\Omega\subset\overline{U},\;\Omega\subset\overline{\co}\big(\{x_0\}\cup F(\Omega)\big)\Longrightarrow\overline{\Omega}\;compact
\end{equation} and
\begin{equation}\label{LS}
x\not\in(1-\lambdaup)x_0+\lambdaup F(x)\text{ on }\overline{U}\setminus U\text{ for all }\lambdaup\in(0,1).
\end{equation}
Then $\fix(F)$ is nonempty and compact. 
\end{theorem}

\begin{proof}
Keeping the notation and notions contained in the proof of \cite[Theorem 3.]{pietkun2}, consider a family $\{M_\alpha\}_{\alpha\in A}$ of all fundamental subsets of there defined multimap $\w{F}\colon X\map X$, containing $x_0$. Recall after Krasnosel'ski\u{\i} that the closed convex set $M\subset X$ is a fundamental subset of $\w{F}$ if $\w{F}(M)\subset M$ and for any $x\in X$, it follows from $x\in\overline{\co}\,\big(\w{F}(x)\cup M\big)$ that $x\in M$. Observe that the family $\{M_\alpha\}_{\alpha\in A}$ is nonempty (take for example $X$). Define $M_0:=\bigcap\limits_{\alpha\in A}M_\alpha$. Next, note that $M_0$ and $\overline{\co}\,\big(\w{F}(M_0)\cup\{x_0\}\big)$ are fundamental. Whence, $M_0=\overline{\co}\,\big(\w{F}(M_0)\cup\{x_0\}\big)$.\par Notice that \[M_0\cap U\subset M_0=\overline{\co}\,\big(\w{F}(M_0)\cup\{x_0\}\big)=\overline{\co}\,\big(F(M_0\cap U)\cup\{x_0\}\big),\] by the very definition of $\w{F}$. Thus, $\overline{M_0\cap U}$ is compact, by \eqref{monch}. Eventually, $M_0$ has compact closure, for the convex envelope $\overline{\co}\,\big(F(M_0\cap U)\cup\{x_0\}\big)$ is compact.\par As we have seen in the proof of \cite[Theorem 3.]{pietkun2}, $\w{F}\colon M_0\map M_0$ is also an admissible multimap. By virtue of the Dugundji Extension Theorem the domain $M_0$ is an absolute extensor for the class of metrizable spaces. Therefore, the set-valued map $\w{F}\colon M_0\map M_0$ must have at least one fixed point $x\in M_0$, in view of Theorem \ref{Lefschetz}. Moreover, $\fix(\w{F})$ forms a closed subset of the compact domain $M_0$ and $\fix(F)=\fix(\w{F})$.
\end{proof}

\begin{theorem}\label{fixed2}
Let $\mathbb{F}$ be a Fr\'echet space and $X$ a nonempty closed convex subset of $\mathbb{F}$. Assume that $U$ is relatively open in $X$ and its closure is a retract of $X$. Assume further that $F\colon\overline{U}\map X$ is either 
\begin{itemize}
\item[(i)] a continuous admissible set-valued map
\end{itemize}
or
\begin{itemize}
\item[(ii)] an upper semicontinuous set-valued map with compact convex values.
\end{itemize}
If for some $x_0\in U$ condition \eqref{LS} together with
\begin{equation}\label{chnom2}
\Omega\subset\overline{U}\,\text{countable},\;\Omega\subset\overline{\co}\,\big(\{x_0\}\cup F(\Omega)\big)\Longrightarrow\overline{\Omega}\;compact
\end{equation} 
is satisfied, then $\fix(F)$ is nonempty and compact. 
\end{theorem}

\begin{proof}
Theorem \ref{fixed2}(i) was proven in \cite{pietkun2}. Theorem \ref{fixed2}(ii) is nothing more than \cite[Theorem 3.2]{regan}.
\end{proof}

\begin{remark}\label{rem}
{\em The following properties of $F\colon\overline{U}\map X$ imply the Leray--Schauder boundary condition \eqref{LS} with $x_0\in U$:
\begin{itemize}
\item[(i)] if $\,\lambdaup(x-x_0)\in F(x)-x_0$ for $x\in\partial U$, then $\lambdaup\<1$ $($Yamamuro's condition$)$,
\item[(ii)] $U$ is convex and $F(\partial U)\subset\overline{U}$ $($Rothe's condition$)$,
\item[(iii)] $|y-x|^2\geqslant |y-x_0|^2-|x-x_0|^2$ for each $x\in\partial U$ and $y\in F(x)$ $($Krasnosel'ski\u{\i}--Altman's condition$)$,
\item[(iv)]  $\langle y-x_0,x-x_0\rangle\<|x-x_0|^2$ for each $x\in\partial U$ and $y\in F(x)$ if $E$ is a Hilbert space $($Browder's condition$)$.
\end{itemize}}
\end{remark}

Assume that $p\in [1,\infty]$ and $q\in[1,\infty)$. Fix a compact segment $I:=[0,T]$ for some end time $T>0$. Let $F\colon I\times E\map E$ be a set-valued map. Throughout the paper we will use the following hypotheses on the mapping $F$:
\begin{itemize}
\item[$(\F_1)$] for every $(t,x)\in I\times E$ the set $F(t,x)$ is nonempty and convex,
\item[$(\F_2)$] the map $F(\cdot,x)$ has a strongly measurable selection for every $x\in E$,
\item[$(\F_3)$] the graph $\Graph(F(t,\cdot))$ is sequentially closed in $(E,|\cdot|)\times(E,w)$ for a.a. $t\in I$,
\item[$(\F_4)$] $F$ satisfies a sublinear growth condition, i.e. there is $b\in L^q(I,\R{})$ and $c>0$ such that for all $x\in E$ and for a.a. $t\in I$,\[||F(t,x)||^+:=\sup\{|y|_E\colon y\in F(t,x)\}\<b(t)+c|x|^{\frac{p}{q}},\]
when $p\in[1,\infty)$. If $p=\infty$, then for every $R>0$ there exists $b_R\in L^q(I,\R{})$ such that \[||F(t,x)||^+\<b_R(t)\;\text{ a.e. on }I, \text{ for all }x\in E\text{ with }|x|\<R.\]
\item[$(\F_5)$] for every closed separable linear subspace $E_0$ of $E$ the map $\res{F}{I\times E_0}(t,\cdot)\cap E_0$ is quasicompact for a.a. $t\in I$.
\end{itemize}
Recall that the Nemtyski\v{\i} operator $N_F\colon L^p(I,E)\map L^q(I,E)$, corresponding to $F$, is a multivalued map defined by
\[N_F(u):=\{w\in L^q(I,E)\colon w(t)\in F(t,u(t))\mbox{ for a.a. }t\in I\}.\] Consider also a multivalued external operator $K\colon L^q(I,E)\map L^p(I,E)$. Our hypothesis on the multifunction $K$ is the following:
\begin{itemize}
\item[$(\es_1)$] for every compact $C\subset E$, the map $K\colon(L^q(I,C),w)\map(L^p(I,E),||\cdot||_p)$ is acyclic,
\item[$(\es_2)$] the map $K\colon L^q(I,E)\map L^p(I,E)$ is $L$-Lipschitz with closed values,
\item[$(\es_3)$] for every uniformly $q$-integrable possessing relatively weakly compact vertical slices a.e. on $I$ subset $C\subset L^q(I,E)$, the map $K\colon\f(C,w\g)\map(L^p(I,E),||\cdot||_p)$ is acyclic,
\end{itemize}
\begin{remark}\label{dyskusja}\mbox{}
{\em \begin{itemize}
\item[(i)] For every relatively weakly compact $C\subset L^q(I,E)$, $K\colon(C,w)\map(L^p(I,E),||\cdot||_p)$ is compact valued upper semicontinuous iff given a sequence $(x_n,y_n)$ in the graph $\Graph(K)$ with $x_n\xrightharpoonup[n\to\infty]{C}x$, there is a subsequence $y_{k_n}\xrightarrow[n\to\infty]{L^p(I,E)}y\in K(x)$. $($notice that the space $(C,w)$ is sequential as a subset of the angelic space $(L^q(I,E),w))$
\item[(ii)] $(\es_3)\Rightarrow(\es_1)$.
\end{itemize}}
\end{remark}

Before we will be able to set forth a result concerning the existence of solutions to inclusion \eqref{hammer}, we have to prove a few auxiliary facts.
\begin{lemma}\label{Niemy}
Let $p\in[1,\infty]$. If the multimap $F\colon I\times E\map E$ satisfies conditions $(\F_1)$--$(\F_5)$, then the Nemytski\v{\i} operator $N_F\colon L^p(I,E)\map L^q(I,E)$ is a weakly upper semicontinuous multivalued map with nonempty convex weakly compact values. 
\end{lemma}
\begin{proof} 
For any $u\in L^p(I,E)$  one can always define a sequence $(u_n)_\n$ of simple~functions, which converges to $u$ almost everywhere and for which $|u_n(t)|\<2|u(t)|$ for every $t\in I$ (cf. the proof of \cite[Theorem III.2.22]{dunford}). Consequently, vertical slices $\x{u_n(t)}$ are relatively compact in $E$ for a.a. $t\in I$. \par Accordingly to the assumption $(\F_2)$ we can indicate a strongly measurable map $w_n\colon I\to E$ such that $w_n(t)\in F(t,u_n(t))$ for a.a. $t\in I$. Thanks to condition $(\F_4)$ we know that the sequence $\z{w_n}$ is $q$-integrably bounded. Let $E_0$ be a closed separable linear subspace of $E$ such that $\x{u_n(t)}\cup\x{w_n(t)}\subset E_0$ a.e. on $I$.
By $(\F_5)$, the vertical slices $\x{w_n(t)}$ are relatively compact a.e. on $I$. In view of Theorem \ref{ulger} the sequence $\z{w_n}$ is relatively weakly compact in $L^q(I,E)$. Hence we may assume, passing to a subsequence if necessary, that $w_n\xrightharpoonup[n\to\infty]{L^q(I,E)}w$. \par Observe that for a.a. $t\in I$, the multimap $F(t,\cdot)$ is compact valued upper semicontinuous. Consider sequences $\z{x_n}$ and $\z{y_n}$ satisfying $x_n\xrightarrow[n\to\infty]{E}x$ and $y_n\in F(t,x_n)$. Put $E_0:=\overline{\text{span}}\f(\x{x_n}\cup\x{y_n}\g)$. Then $\x{y_n}\subset F\f(t,\overline{\x{x_n}}\g)\cap E_0$. The latter is relatively compact, in view of $(\F_5)$. Thus, there exists $\z{y_{k_n}}$ such that $y_{k_n}\xrightarrow[n\to\infty]{E}y$. Assumption $(\F_3)$ implies $y\in F(t,x)$.\par Applying Theorem \ref{convergence} one gets $w(t)\in F(t,u(t))$ for a.a. $t\in I$. In this way we have shown that the Nemytski\v{\i} operator $N_F$ has nonempty values. \par Applying similar reasoning one may prove that given a sequence $\z{u_n,w_n}$ in the graph $\Graph(N_F)$ with $u_n\xrightarrow[n\to\infty]{L^p(I,E)}u$, there is a subsequence $w_{k_n}\xrightharpoonup[n\to\infty]{L^q(I,E)}w\in~N_F(u)$. Indeed, since the family $\x{|u_n(\cdot)|^p}$ is uniformly integrable, the sequence $\z{w_n}$ must be uniformly $q$-integrable as well ($q$-integrably bounded in case of $p=\infty$). Therefore, we may apply weak compactness criterion (Theorem \ref{ulger}) to extract a subsequence $\z{w_{k_n}}$ with $w_{k_n}\xrightharpoonup[n\to\infty]{L^q(I,E)}w$. Convergence theorem (Theorem \ref{convergence}) entails $w\in N_F(u)$.
\end{proof}

\begin{corollary}\label{niemycor}
Let $E$ be a reflexive Banach space space. If the set-valued map $F\colon I\times E\map E$ fulfills conditions $(\F_1)$-$(\F_4)$, then the thesis of Lemma \ref{Niemy} holds.
\end{corollary}
\begin{proof}
By $(\F_4)$ the map $F(t,\cdot)$ is locally bounded a.e. on $I$. Consider a sequence $(x_n,y_n)_\n$ in the graph $\Graph(F(t,\cdot))$ with $x_n\to x$ in the norm of $E$. Since $E$ is reflexive, there must be a subsequence $y_{k_n}\rightharpoonup y$. Bearing in mind $(\F_3)$, i.e. that $\Graph(F(t,\cdot))$ is strongly-weakly closed, we obtain $y\in F(t,x)$. Therefore, $F(t,\cdot)$ is weakly upper semicontinuous and possesses weakly compact values for a.a. $t\in I$. \par Retaining the notation of the previous proof it may be observed that $\x{w_n(t)}$ forms a subset of a weakly compact set $F\f(t,\overline{\x{u_n(t)}}\g)$ for a.a. $t\in I$. Now, in manner fully analogous to the mentioned proof, we can use Theorem \ref{ulger} and Theorem \ref{convergence} to justify the thesis.
\end{proof}

\begin{lemma}\label{usc}
Assume $(\es_1)$ and $(\es_2)$ are satisfied. Let $M$ be a countable subset of $L^q(I,E)$ such that $M(t)$ is relatively compact in $E$ for a.a. $t\in I$. Then the image $K(M)$ is relatively compact in $L^p(I,E)$ and $K$ is upper semicontinuous from $M$ furnished with the relative weak topology of $L^q(I,E)$ to $L^p(I,E)$ with its norm topology.
\end{lemma}
\begin{proof}
\par Let $M:=\x{w_n}$. In view of Pettis measurability theorem there exists a closed separable subspace $E_0$ of $E$ with $w_n(t)\in E_0$ a.e. on $I$ for every $\n$. For each $k\geqslant 1$ there is a $k$-dimensional linear subspace $E_k\subset E_0$ such that $E_k\subset E_{k+1}$ and $E_0=\overline{\bigcup_{k\geqslant 1}E_k}$. Let $\rho>0$ be an arbitrarily chosen scalar. Obviously, there exists a subset $\Theta_1\subset I$ such that $|w_n(t)|\<\rho$ for every $t\in I\setminus\Theta_1$ and $\n$. Consequently, 
\[\forall\,\n\;\forall\,\K\;\forall\,t\in I\setminus\Theta_1\;\;\;\dist(w_n(t),E_k)=\dist(w_n(t),D_E(0,2\rho)\cap E_k).\]
Take $\eps>0$. Using the formula for the Hausdorff MNC in a separable Banach space \cite[Proposition 2]{harten} one sees that $\beta_{E_0}\f(\x{w_n(t)}\g)=\lim\limits_{k\to\infty}\sup\limits_\n \dist(w_n(t),E_k)$. Thus, in view of Egorov's theorem, 
\[\exists\,\Theta_2\subset I\setminus\Theta_1\,\exists\,k_0\in\varmathbb{N}\;\forall\,k\geqslant k_0\;\forall\,t\in I\setminus(\Theta_1\cup\Theta_2)\;\sup_\n \dist(w_n(t),D_E(0,2\rho)\cap E_k)<\beta_{E_0}\f(\x{w_n(t)}\g)+\frac{\eps}{3}.\] Referring once more to Egorov's theorem we can indicate a measurable $\Theta_3\subset I$ and a simple function $\w{w}_n\colon I\to E$ such that \[|w_n(t)-\w{w}_n(t)|<\frac{\eps}{3}\;\text{ and }\;\dist(\w{w}_n(t),D_E(0,2\rho)\cap E_k)<\frac{2}{3}\eps\] for all $t\in I\setminus(\Theta_1\cup\Theta_2\cup\Theta_3)$, $k\geqslant k_0$ and $\n$. The latter property comes down eventually to the following:
\begin{equation}\label{akuku}
\forall\,\eps>0\;\exists\,k_0\in\varmathbb{N}\;\forall\,k\geqslant k_0\;\forall\,\n\;\exists\,w_{n,k}\in L^q(I,D_E(0,2\rho)\cap E_k)\text{ such that }||w_n-w_{n,k}||_q<\eps.
\end{equation}
Fix $\eps>0$ and $k\geqslant k_0$. The range $\x{w_{n,k}}$ is relatively weakly compact in $L^q(I,E)$ in view of Theorem \ref{ulger}. Condition $(\es_1)$ implies the relative compactness of $K\f(\x{w_{n,k}}\g)$ in $L^p(I,E)$. Thus, making use of \eqref{akuku} and $(\es_2)$ we arrive at 
\begin{align*}
\f|\beta_{L^p}\f(K\f(\{w_n\}_{n=1}^\infty\g)\g)\g|&=\f|\beta_{L^p}\f(K\f(\{w_n\}_{n=1}^\infty\g)\g)-\beta_{L^p}\f(K\f(\{w_{n,k}\}_{n=1}^\infty\g)\g)\g|\<h\f(\bigcup_{n=1}^\infty K(w_n),\bigcup_{n=1}^\infty K(w_{n,k})\g)\\&\<\sup_\n h(K(w_n),K(w_{n,k}))\<\sup_\n L||w_n-w_{n,k}||_q\<\eps L.
\end{align*}
Since $\eps$ was arbitrary, the image $K\f(\{w_n\}_{n=1}^\infty\g)$ must be relatively compact.\par Assume that $(w_n,v_n)\in\Graph(K)$ with $w_n\xrightharpoonup[n\to\infty]{M}w$. As we have shown above the set $K(\x{w_n})$ is relatively compact. Thus, there exists a subsequence (again denoted by) $\z{v_n}$ such that $v_n\xrightarrow[n\to\infty]{L^p(I,E)}v$. Our aim is to show that $v\in K(w)$. Take $\eps>0$. As previously, we can indicate a sequence $(w_n^\eps)_{n=1}^\infty$ and a compact subset $C_\eps\subset E$ such that $\{w_n^\eps\}_{n=1}^\infty\subset L^q(I,C_\eps)$ and $||w_n-w_n^\eps||_q\<\frac{\eps}{4L}$. In view of the weak compactness criterion (Theorem \ref{ulger}), we may assume that $w_n^\eps\xrightharpoonup[n\to\infty]{L^q(I,E)}w^\eps$, passing once again to a subsequence if necessary. Clearly, $||w-w^\eps||_q\<\frac{\eps}{4L}$ due to the weak lower semicontinuity of the norm. Choose $y_n^\eps\in K(w_n^\eps)$ in such a way that $||v_n-y_n^\eps||_p=\dist(v_n,K(w_n^\eps))$. Assumption $(\es_1)$ guarantees that $y_n^\eps\xrightarrow[n\to\infty]{L^p(I,E)}y^\eps\in K(w^\eps)$, up to a subsequence. Of course, there is $N\in\varmathbb{N}$ such that $||v_N-v||_p\<\frac{\eps}{4}$ and $||y_N^\eps-y^\eps||_p\<\frac{\eps}{4}$. Now, it is possible to estimate
\begin{align*}
\dist(v,K(w))&\<||v-v_N||_p+||v_N-y_N^\eps||_p+||y_N^\eps-y^\eps||_p+\dist(y^\eps,K(w))\<\frac{\eps}{4}+\dist(v_N,K(w_N^\eps))+\frac{\eps}{4}\\&+h(K(w^\eps),K(w))\<\frac{\eps}{2}+h(K(w_N),K(w_N^\eps))+h(K(w^\eps),K(w))\\&\<\frac{\eps}{2}+L||w_N-w_N^\eps||_q+L||w-w^\eps||_q\<\eps. 
\end{align*}
Since $\eps$ was arbitrary, it follows that $v\in K(w)$.
\end{proof}

\begin{lemma}\label{acyc2}
Let $X$ be a compact topological space and $Y$ be a paracompact topological space. Assume that $F\colon X\map Y$ is an upper semicontinuous surjective multimap with compact acyclic values and acyclic fibers. Then there is an isomorphism $H^\ast(X)\approx H^\ast(Y)$.
\end{lemma}
\begin{proof}
Since $X$ is compact, the product $X\times Y$ is a paracompact space. The space $Y$ is regular and the map $F$ is usc so the graph $\Graph(F)$ is a closed subset of $X\times Y$. Thus it is also a paracompact space. The projection $\pi_1\colon\Graph(F)\to X$ of $\Graph(F)$ onto the domain $X$ is continuous and surjective. It is easy to see that $\pi_1$ is a closed map, since $F$ is compact valued and upper semicontinuous. Moreover, the fibers $\pi_1^{-1}(\{x\})=\{x\}\times F(x)$ are compact acyclic. Hence $\pi_1$ is perfect and consequently a proper map. Analogously, the projection $\pi_2\colon\Graph(F)\to Y$ is surjective continuous and the preimage $\pi_2^{-1}(\{y\})=F^{-1}(\{y\})\times\{y\}$ is compact acyclic. The map $\pi_2$ is also closed, since the domain $X$ is compact. In wiev of Vietoris--Begle mapping theorem \cite[Theorem 6.9.15]{spanier} it follows that $(\pi_1^*)^{-1}\circ\pi_2^*\colon H^\ast(Y)\to H^\ast(X)$ is an isomorphism.
\end{proof}
Recall that for the sake of convenience we had introduced the letter $H$ to denote the superposition $K\circ N_F$.
\begin{lemma}\label{usc2}
Let $(\F_1)$--$(\F_4)$ be satisfied. Assume that either $E$ is reflexive and $(\es_3)$ holds or $(\es_1)$--$(\es_2)$ and $(\F_5)$ are met. In both cases, the operator $H\colon L^p(I,E)\map L^p(I,E)$ is a compact valued upper semicontinuous map.
\end{lemma}
\begin{proof}
Assume that $u_n\xrightarrow[n\to\infty]{L^p(I,E)}u$. Obviously, there exists a subsequence $\z{u_{k_n}}$ such that $u_{k_n}(t)\xrightarrow[n\to\infty]{E}u(t)$ for a.a. $t\in I$. Let $w_n\in N_F(u_n)$ and $v_n\in K(w_n)$ for $\n$. Observe that the sequence $\z{w_n}$ is bounded uniformly $q$-integrable (or simply bounded for $p=\infty$).\par If $E$ is reflexive, then the map $F(t,\cdot)$ is weakly upper semicontinuous and possesses weakly compact values a.e. on $I$. Thus, the sets $\x{w_{k_n}(t)}$ are relatively weakly compact for a.a. $t\in I$. The sequence $\z{w_{k_n}}$ is relatively compact in view of Theorem \ref{ulger}. We may assume, passing again to a subsequence if necessary, that $w_{k_n}\xrightharpoonup[n\to\infty]{L^q(I,E)}w$. Condition $(\es_3)$ implies (cf. Remark \ref{dyskusja}) that $v_{k_n}\xrightarrow[n\to\infty]{L^p(I,E)}v\in K(w)$, up to a subsequence. It is enough to apply Corollary \ref{niemycor} to show that $w\in N_F(u)$. Eventually, $v\in H(u)$, i.e. the set-valued map $H\colon L^p(I,E)\map L^p(I,E)$ is an upper semicontinuous operator with compact values.\par If assumption $(\F_5)$ is met, then the multimap $F(t,\cdot)$ is compact valued and upper semicontinuous a.e. on $I$. In this case the sets $\x{w_{k_n}(t)}$ are relatively compact in $E$ for a.a. $t\in I$. Passing to a subsequence if necessary, we obtain $w_{k_n}\xrightharpoonup[n\to\infty]{L^q(I,E)}w$. By virtue of Lemma \ref{usc} there exists a subsequence (again denoted by) $\z{v_{k_n}}$ such that $v_{k_n}\xrightarrow[n\to\infty]{L^p(I,E)}v\in K(w)$. Lemma \ref{Niemy} implies that $w\in N_F(u)$. This means that $v\in H(u)$.
\end{proof}

\begin{lemma}\label{(3)}
Let $U\subset L^p(I,E)$ and $x_0\in U$. Assume that $F\colon I\times E\map E$ satisfies $(\F_1)$--$(\F_4)$. Suppose further that operator $H\colon\overline{U}\map L^p(I,E)$ with uniformly $p$-integrable range $($or bounded range if $p=\infty)$ meets the following condition:
\begin{equation}\label{monch3}
\f(M\subset\overline{U}\text{ and }\overline{M}\subset\overline{\co}\,\big(\{x_0\}\cup H(M)\big)\g)\Longrightarrow M(t)\text{ is relatively compact for a.a. }t\in I.
\end{equation}
Assume also that either $E$ is reflexive and $(\es_3)$ holds or $(\es_1)$--$(\es_2)$ and $(\F_5)$ are met. In both the cases, condition \eqref{monch} is fulfilled.
\end{lemma}
\begin{proof}
Assume that $M\subset\overline{U}$ and $\overline{M}\subset\overline{\co}\,\big(\{x_0\}\cup H(M)\big)$. Consider $\x{v_n}\subset H(M)$. Let $v_n\in K(w_n)$ with $w_n\in N_F(M)$. By \eqref{monch3}, vertical slices $M(t)$ are relatively compact for a.a. $t\in I$. \par Suppose $E$ is reflexive. Taking into account that $\x{w_n(t)}\subset F(t,\overline{M(t)})$ a.e. on $I$ and that $F(t,\cdot)$ is weakly upper semicontinuous we see that $\x{w_n(t)}$ is relatively weakly compact for a.a. $t\in I$. Since $M$ forms a subset of uniformly $p$-integrable convex hull $\co\,\big(\{x_0\}\cup H(M)\big)$, the sequence $\z{w_n}$ must be uniformly $q$-integrable. It follows from condition $(\es_3)$ that the image $K\big(\overline{\x{w_n}}^w\big)$ is compact in $L^p(I,E)$. Hence the set $\x{v_n}$ is relatively compact. The latter entails the relative compactness of $\overline{\co}\,\big(\{x_0\}\cup H(M)\big)$ and eventually the compactness of the closure $\overline{M}$.\par  If conditions $(\es_1)$--$(\es_2)$ and $(\F_5)$ are met, then the map $F(t,\cdot)$ is upper semicontinuous, vertical slices $\x{w_n(t)}$ are relatively compact a.e. on $I$ and the image $K\big(\x{w_n}\big)$ is relatively compact in the space $L^p(I,E)$ in view of Lemma \ref{usc}. Therefore, $\overline{M}$ is a compact subset of $L^p(I,E)$.
\end{proof}

\begin{remark}
{\em Clearly, the operator $H\colon\overline{U}\map L^p(I,E)$, which is condensing relative to some monotone nonsingular and regular MNC $\gamma$ defined on the space $L^p(I,E)$, satisfies condition \eqref{monch}.}
\end{remark}

The eponymous solvability of operator inclusions of Hammerstein type expresses itself in the following fixed point principle, formulated in the context of the Bochner space $L^p(I,E)$.

\begin{theorem}\label{solvability}
Let $X$ be a closed convex subset of the space $L^p(I,E)$. Assume that either
\begin{itemize} 
\item[(i)] the space $E$ is reflexive, the operator $K\colon L^q(I,E)\map X$ possesses convex or decomposable fibers and satisfies assumption $(\es_3)$, the multimap $F\colon I\times E\map E$ meets conditions $(\F_1)$--$(\F_4)$\par\noindent or
\item[(ii)] the operator $K\colon L^q(I,E)\map X$ possesses compact acyclic values and convex or decomposable fibers and satisfies assumptions $(\es_1)$--$(\es_2)$, the set-valued map $F\colon I\times E\map E$ meets conditions $(\F_1)$--$(\F_5)$.
\end{itemize}
Suppose further that there exists a radius $R>0$ such that
\begin{equation}\label{radius1}
L\f(||b||_q+c\f(R+||K(0)||_p^+\g)^\frac{p}{q}\g)\<R
\end{equation}
if $p<\infty$ and respectively
\begin{equation}\label{radius2}
L\f\Arrowvert b_{{\textstyle R+||K(0)||_\infty^+}}\g\Arrowvert_q\<R
\end{equation}
provided $p=\infty$. If the operator $H\colon D_{L^p}(K(0),R)\cap X\map X$ with uniformly $p$-integrable range satisfies condition \eqref{monch3} with $x_0\in K(0)$, then there exists at least one solution $x\in D_{L^p}(K(0),R)\cap X$ of the initial inclusion~\eqref{hammer}.
\end{theorem}
\begin{proof}
\par Fix $u\in L^p(I,E)$. The subset $N_F(u)$ furnished with the relative weak topology of $L^q(I,E)$ is compact (cf. Lemma \ref{Niemy} or Corollary \ref{niemycor}). Moreover, $(N_F(u),w)$ is in fact an acyclic space, given that $N_F(u)$ is always contractible in the weak topology $\sigma(L^q(I,E),L^p(I,E^*))$ (regardless of whether the values of $F$ are convex or not, because values of the Nemytski\v{\i} operator are still decomposable). Under assumption $(\es_3)$ the multimap $K\colon\f(N_F(u),w\g)\map(H(u),||\cdot||_p)$ may be regarded as an acyclic operator between compact topological space $(N_F(u),w)$ and a paracompact space $(H(u),||\cdot||_p)$. The same can be said if we assume that $K$ is Lipschitz with compact acyclic values. Observe that the intersection $K^{-1}(\{v\})\cap N_F(u)$ is convex in case $K$ has convex fibers or decomposable if we assume that the fibers of $K$ are decomposable. Therefore, the fibers of the mutimap under consideration are acyclic. In view of Lemma \ref{acyc2} we are allowed to conclude that the reduced Alexander--Spanier cohomologies $\w{H}^\ast((H(u),||\cdot||_p))$ are trivial, i.e. the image $H(u)$ is acyclic as a subspace of $L^p(I,E)$.\par From Lemma \ref{usc2} follows that the operator $H\colon D_{L^p}(K(0),R)\cap X\map X$ is compact valued upper semicontinuous. Considering what we have established so far, it is apparent that $H$ is an acyclic operator. Lemma \ref{(3)} guarantees that condition \eqref{monch} is also satisfied. \par Let $p<\infty$ and $R>0$ be matched according to \eqref{radius1}. Take $u\in D_{L^p}(K(0),R)$. Since $K(0)$ is compact (both in the case (i) and in the case (ii)), there is $z_u\in K(0)$ such that $||u-z_u||_p=\dist(u,K(0))$. Observe that
\[||N_F(u)||_q^+\<||b||_q+c||u||_p^\frac{p}{q}\<||b||_q+c\f(||u-z_u||_p+||z_u||_p\g)^\frac{p}{q}\<||b||_q+c\f(R+||K(0)||^+_p\g)^\frac{p}{q}.\] Whence \[\sup_{v\in H(u)}\dist(v,K(0))\<\sup_{w\in N_F(u)}h(K(w),K(0))\<\sup_{w\in N_F(u)}L\,||w||_q=L\,||N_F(u)||^+_q\<L\f(||b||_q+c\f(R+||K(0)||^+_p\g)^\frac{p}{q}\g).\] Eventually, \[H\f(D_{L^p}(K(0),R)\cap X\g)\subset D_{L^p}(K(0),R)\cap X,\] by \eqref{radius1}. The latter entails \eqref{LS}. Indeed, fix any $x_0\in K(0)$, $\lambdaup\in(0,1)$ and $x\in D_{L^p}(K(0),R)$. Then \[\sup_{v\in H(x)}\dist((1-\lambdaup)x_0+\lambdaup v,K(0))\<\lambdaup\sup_{v\in H(x)}\dist(v,K(0))\<\lambdaup R<R.\] Thus, $x\not\in(1-\lambdaup)x_0+\lambdaup H(x)$ provided $x\in\partial D_{L^p}(K(0),R)$. In analogous manner one can show that Leray-Schauder boundary condition \eqref{LS} is satisfied under assumption \eqref{radius2}.\par  In view of Theorem \ref{fixed}. we infer that the multifunction $H$ has a fixed point in $D_{L^p}(K(0),R)\cap X$.
\end{proof}
\begin{remark}
{\em As it comes to formulation of sufficient conditions for acyclicity of the values of the superposition $K\circ N_F$ $($cf. \cite[Remark 4.2]{precup}$)$, it should be emphasized that condition: for all $w_0,w_1,w_2\in L^q(I,E)$, the equality $K(w_1)=K(w_2)$ implies \[K\f(w_1{\bf 1}_{[0,\lambdaup]}+w_0{\bf 1}_{[\lambdaup,T]}\g)=K\f(w_2{\bf 1}_{[0,\lambdaup]}+w_0{\bf 1}_{[\lambdaup,T]}\g)\] for every $\lambdaup\in I$, is much stronger than assumption regarding the decomposability of the fibers of operator $K$. Similarly, the condition that operator $K$ is affine is visibly stronger than the fact that $K$ has convex fibers.}
\end{remark}

\section{Examples}
We conclude this paper with examples, which illustrate the wide range of applications of the unified topological approach, developed in the previous section, to integro-differential inclusions.\\\par
\noindent{\bf Example 1.} Given an $m$-accretive operator $A\colon D(A)\subset E\map E$ in a Banach space $E$ and a multivalued perturbation $F\colon I\times\overline{\co}D(A)\map E$ we consider the initial value problem: 
\begin{equation}\label{cauchy}
\begin{gathered}
\dot{u}(t)\in -Au(t)+F(t,u(t))\;\;\;\text{on }I,\\
u(0)=u_0.
\end{gathered}
\end{equation}
If $A$ is $m$-accretive and $U(\cdot)x$ is an integral solution of \eqref{cauchy} with $F\equiv 0$ and $u(0)=x$, then the family $\{U(t)\}_{t\geqslant 0}$ of nonexpansive mappings $U(t)\colon\overline{D(A)}\to\overline{D(A)}$ is called the semigroup generated by $-A$.
\begin{theorem}
Let $E^*$ be strictly convex and $A\colon D(A)\subset E\map E$ be $m$-accretive such that $-A$ generates an equicontinuous semigroup. Assume that $F\colon I\times\overline{\co}\,D(A)\map E$ satisfies $(\F_1)$-$(\F_3)$ together with 
\begin{itemize}
\item[$(\F_4')$] there is $\mu\in L^1(I,\R{})$ such that $||F(t,x)||^+\<\mu(t)(1+|x|)$ for all $x\in E$ and for a.a. $t\in I$ 
\end{itemize}
and
\begin{itemize}
\item[$(\F_6)$] there is a function $\eta\in L^1(I,\R{})$ such that for all bounded subsets $\Omega\subset E$ and for a.a. $t\in I$ the inequality holds \[\beta(F(t,\Omega))\<\eta(t)\beta(\Omega).\]
\end{itemize}
Then the Cauchy problem \eqref{cauchy} has an integral solution for every $u_0\in\overline{D(A)}$.
\end{theorem}
\begin{proof}
As it is well known, we can associate with any $w\in L^1(I,E)$ a unique integral solution $S(w)\in C(I,\overline{D(A)})$ of the quasi-autonomous problem
\begin{equation}\label{cauchy2}
\begin{gathered}
\dot{u}(t)\in -Au(t)+w(t)\;\;\;\text{on }I,\\
u(0)=u_0.
\end{gathered}
\end{equation}
The mapping $S\colon L^1(I,E)\to C(I,E)$ satisfies \[|S(w_1)(t)-S(w_2)(t)|\<|S(w_1)(s)-S(w_2)(s)|+\int_s^t|w_1(\tau)-w_2(\tau)|\,{\rm d}\tau\] for all $0\<s\<t\<T$, which means in particular that $S$ meets condition $(\es_2)$.\par Take $w_1,w_2\in S^{-1}(\{u\})$ and fix $\lambdaup\in(0,1)$. For every $(x,y)\in\Graph(A)$ and $0\<s\<t\<T$ the following inequality holds \[|u(t)-x|^2-|u(s)-x|^2\<2\int_s^t\langle w_i(\tau)-y,u(\tau)-x\rangle_+\,{\rm d}\tau,\] where $i=1$ and $i=2$ for the cases $S(w_1)$ and $S(w_2)$, respectively. Since $E^*$ is strictly convex, the semi-inner products are indistinguishable, i.e. $\langle x,y\rangle_+=\langle x,y\rangle_-$. In view of the latter we are allowed to write down the following estimation:
\begin{align*}
|u(t)-x|^2-|u(s)-x|^2&\<\lambdaup\,2\!\int_s^t\langle w_1(\tau)-y,u(\tau)-x\rangle_+\,{\rm d}\tau+(1-\lambdaup)\,2\!\int_s^t\langle w_2(\tau)-y,u(\tau)-x\rangle_+\,{\rm d}\tau\\&=2\int_s^t\lambdaup\langle w_1(\tau)-y,u(\tau)-x\rangle_++(1-\lambdaup)\langle w_2(\tau)-y,u(\tau)-x\rangle_+\,{\rm d\tau}\\&=2\int_s^t\langle(\lambdaup w_1+(1-\lambdaup)w_2)(\tau)-y,u(\tau)-x\rangle_+\,{\rm d}\tau.
\end{align*}
This means that $u$ constitutes a solution to the quasi-autonomous problem
\begin{equation*}
\begin{gathered}
\dot{u}(t)\in -Au(t)+(\lambdaup w_1+(1-\lambdaup)w_2)(t)\;\;\;\text{on }I,\\
u(0)=u_0.
\end{gathered}
\end{equation*}
In other words $u=S(\lambdaup w_1+(1-\lambdaup)w_2)$, i.e. the fiber $S^{-1}(\{u\})$ is convex.\par Consider a compact subset $C\subset E$ and a sequence $\z{w_n}\subset L^1(I,C)$ such that $w_n\xrightharpoonup[n\to\infty]{L^1(I,E)}w$. Since $-A$ generates an equicontinuous semigroup and the family $\x{w_n}$ is uniformly integrable, the image $S(\x{w_n})\subset C(I,E)$ is equicontinuous. This is the thesis of \cite[Theorem 2.3]{gutman}. In view of Pettis measurability theorem $E_0:=\overline{\spn}\bigcup_{n=1}^\infty(w_n(I))$ is a separable subspace of $E$. The arguments contained in the proof of part (b) of \cite[Lemmma 4]{bothe} justify the following estimate \[\beta\f(\x{S(w_n)(t)}\g)\<\int_0^t\beta_{E_0}\f(\x{w_n(s)}\g)\,{\rm d}s\] for $t\in I$. It should be stressed here that the veracity of this formula is completely independent of the geometrical properties of the dual space $E^*$, such as the uniform convexity assumed by the author of \cite{bothe}. As a consequence, we get that $\beta\f(\x{S(w_n)(t)}\g)=0$ for all $t\in I$. In view of the Arzel\`a theorem the sequence $\z{Sw_n}$ must be uniformly convergent to some $v$. The extra condition regarding the geometry of the dual space $E^*$ makes it possible to demonstrate that $S(w)=v$ (cf. \cite{tolst}). Therefore, operator $S$ meets condition $(\es_1)$.\par Let $R_0:=\sup_{t\in I}|U(t)x_0|$. To indicate a priori bounds on the solutions of \eqref{cauchy} consider $w\in N_F(S(w))$. It is easy to see that \[|S(w)(t)|\<|S(0)(t)|+|S(w)(t)-S(0)(t)|\<|U(t)x_0|+\int_0^t|w(s)|\,{\rm d}s\<R_0+||\mu||_1+\int_0^t\mu(s)|S(w)(s)|\,{\rm d}s\] for $t\in I$. Hence, \[||S(w)||\<(R_0+||\mu||_1)e^{||\mu||_1}=:M\] for any $w\in\fix(N_F\circ S)$, by the Gronwall inequality. Now, the standard trick allows us to assume that $||F(t,x)||^+\<\mu(t)(1+M)=\delta(t)$ a.e. on $I$ with $\delta\in L^1(I,\R{}_+)$. Otherwise we may always replace the right-hand side $F$ by $F(\cdot,r(\cdot))$ with $r\colon E\to D_E(0,M)\cap\overline{\co}D(A)$ being a retraction.\par Let $X:=L^\infty(I,\overline{\co}\,D(A))$. Clearly, $X$ is closed and convex in $L^\infty(I,E)$. Define $H\colon X\map X$ to be the superposition $H:=S\circ N_F$, where $N_F\colon X\map L^1(I,E)$ is the Nemytski\v{\i} operator corresponding to the righ-hand side $F$. According to the above observations on the growth of $F$, if $w\in L^1(I,E)$ then \[||S(w)-S(0)||\<|S(w)(0)-S(0)(0)|+\int_0^T|w(s)|\,{\rm d}s\<||\delta||_1=:R.\] This means, in particular, that $H(\partial B_{L^\infty}(S(0),R)\cap X)\subset D_{L^\infty}(S(0),R)\cap X$. Put $U:=B_{L^\infty}(S(0),R)\cap X$. It follows that operator $H\colon\overline{U}\map X$ satisfies boundary condition \eqref{LS} with $x_0:=S(0)\in U$.\par We claim that operator $H$ meets condition \eqref{monch3}. Let $M\subset\overline{U}$ be such that $\overline{M}\subset\overline{\co}\,\big(\{S(0)\}\cup H(M)\big)$. Whence, $\overline{M}(t)\subset\overline{\co}\,\big(\{S(0)(t)\}\cup H(M)(t)\big)$ a.e. on $I$. Since $N_F(M)$ is uniformly integrable, the image $S(N_F(M))$ is equicontinuous (cf. \cite[Theorem 2.3]{gutman}). Hence, the mapping $I\ni t\mapsto\beta(H(M)(t))\in\R{}_+$ must be continuous. Fix $t\in I$. There exists $\x{w_n^t}\subset N_F(M)$ such that \[\beta\big(\x{S(w_n^t)(t)}\big)=\max\{\,\beta(D)\colon D\subset H(M)(t)\;\text{countable}\}.\] Let $E_0$ be a closed and separable subspace of $E$ such that $\x{w_n^t(\tau)}\subset E_0$ for a.a. $\tau\in I$. Under assumption $(\F_6)$ the following estimate is easily verifiable: 
\begin{align*}
\beta_{E_0}(\x{w_n^t(\tau)})&\<2\,\beta(\x{w_n^t(\tau)})\<2\,\beta(F(\tau,\x{u_n^t(\tau)}))\<2\eta(\tau)\,\beta(\x{u_n^t(\tau)})\<2\eta(\tau)\,\beta(M(\tau))\\&\<2\eta(\tau)\,\beta(H(M)(\tau))
\end{align*}
for a.a. $\tau\in I$. As we have noticed previously, the proof of \cite[Lemmma 4]{bothe} constitutes a justification for the estimation \[\beta\big(\x{S(w_n^t)(\tau)}\big)\<\int_0^\tau\beta_{E_0}\big(\x{w_n^t(s)}\big)\,{\rm d}s.\] for $\tau\in I$. Hence, 
\begin{align*}
\beta(H(M)(t))&\<2\max\{\,\beta(D)\colon D\subset H(M)(t)\;\text{countable}\}=2\,\beta\big(\x{S(w_n^t)(t)}\big)\<2\int_0^t\beta_{E_0}\big(\x{w_n^t(s)}\big)\,{\rm d}s\\&\<4\int_0^t\eta(s)\beta(H(M)(s))\,{\rm d}s.
\end{align*} 
Since the time $t$ was chosen arbitrary, we infer that \[\beta(H(M)(t))\<4\int_0^t\eta(s)\beta(H(M)(s))\,{\rm d}s\] on $I$. Consequently, $\sup\limits_{t\in I}\beta(H(M)(t))=0$ and vertical slices $M(t)$ are relatively compact in $E$ for a.a $t\in I$. \par Combining the theses of Lemma \ref{usc2} and \ref{(3)} with the argument taken from the proof of Theorem \ref{solvability}, we infer that operator $H\colon\overline{U}\map X$ meets all the requirements imposed by Theorem \ref{fixed}. Therefore, $H$ possesses a fixed point $u\in\overline{U}\cap C(I,E)$. Clearly, $u$ is an integral solution to \eqref{cauchy}.
\end{proof}
\begin{remark}
{\em Note that our result corresponds exactly to the content of \cite[Theorem 2]{bothe}, except for the assumption about the geometry of the dual space $E^*$. As we have seen it was enough to assume that $E^*$ is strictly convex. The counter-example given in \cite{bothe} $($cf. \cite[Example 1]{bothe}$)$ shows that this geometric condition on $E^*$ cannot be removed.}
\end{remark}

\noindent{\bf Example 2.} In this example we consider the following semilinear nonlocal Cauchy problem
\begin{equation}\label{inclusion}
\begin{gathered}
\dot{x}(t)\in Ax(t)+F(t,x(t))\;\;\;\text{on } I,\\
x(0)=g(x),
\end{gathered}
\end{equation}
where the linear operator $A\colon D(A)\subset E\to E$ is closed densely defined and $g\colon C(I,E)\to E$ is continuous. It is well known that if $A$ poses the infinitesimal generator of a $C_0$-semigroup $\{U(t)\}_{t\geqslant 0}$ of bounded linear operators on $E$, then there are constants $M>0$ and $\omega\in\mathbb{R}$ such that $||U(t)||_{\mathcal L}\<Me^{\omega t}$ for any $t\geqslant 0$. Renorming the Banach space $E$ in an appropriate way one can achieve that $M=1$.\par Let us remind that a Banach space $E$ has property ${\mathcal S}$ if for all separable subspaces $X\subset E$ there exists a separable subspace $Y$ of $E$ with $X\subset Y$ and a continuous linear projection $P\colon E\to Y$ of norm $1$. In particular, weakly compactly generated Banach spaces have property ${\mathcal S}$.

\begin{theorem}\label{semilinear}
Let $A$ be an infinitesimal generator of a $C_0$-semigroup $\{U(t)\}_{t\geqslant 0}$. Assume that $F\colon I\times E\map E$ satisfies $(\F_1)$--$(\F_3)$ together with $(\F_4')$ and $(\F_6)$, whereas $g\colon C(I,E)\to E$ is $L$-Lipschitz and there is $k>0$ such that
\begin{equation}\label{k}
\beta(g(\Omega))\<k\sup_{t\in I}\beta(\Omega(t))
\end{equation}
for every bounded $\Omega\subset C(I,E)$. Assume also that there exists a constant $r>0$ such that 
\begin{equation}\label{R}
0<||\mu||_1<M^{-1}\ln\f(\frac{1+r}{1+M(Lr+|g(0)|)}\g)
\end{equation}
with $M:=\sup\limits_{t\in I}||U(t)||_{\mathcal L}$. Then the boundary value problem \eqref{inclusion} possesses a mild solution in each of the following cases:
\begin{itemize}
\item[(i)] $M(k+||\eta||_1)<1$ if $E$ is has property ${\mathcal S}$,
\item[(ii)] $M(k+4||\eta||_1)<1$ if $E$ is an arbitrary Banach space.
\end{itemize}
\end{theorem}

\begin{proof}
Define $K\colon L^1(I,E)\to C(I,E)$ to be the map which assigns to each $f\in L^1(I,E)$ the unique mild solution of the following problem 
\begin{equation}\label{equ3}
\begin{gathered}
\dot{x}(t)=Ax(t)+f(t),\;\text{a.e. on } I,\\
x(0)=g(x).
\end{gathered}
\end{equation}
$K$ is a well defined single-valued mapping. Indeed, suppose $x,y$ are solutions of \eqref{equ3}. Then \[|x(t)-y(t)|=\f|U(t)g(x)+\int_0^tU(t-s)f(s)\,{\rm d}s-U(t)g(y)-\int_0^tU(t-s)f(s)\,{\rm d}s\g|\<M|g(x)-g(y)|\] and consequently $||x-y||\<ML||x-y||$. From \eqref{R} it follows that $ML<1$ i.e., $x=y$.\par  Let $f,h\in L^1(I,E)$. Then 
\begin{align*}
|K(f)(t)-K(h)(t)|&\<|U(t)g(K(f))+U(t)g(K(h))|+\f|\int_0^tU(t-s)f(s)\,{\rm d}s-\int_0^tU(t-s)h(s)\,{\rm d}s\g|\\&\<ML||K(f)-K(h)||+M\int_0^t|f(s)-h(s)|\,{\rm d}s
\end{align*}
for $t\in I$. Whence \[||K(f)-K(g)||\<\frac{M}{1-ML}||f-g||_1,\] i.e. operator $K$ is Lipschitz. \par We will show that operator $K$ satisfies also condition $(\es_1)$. To see this consider a compact subset $C\subset E$ and a sequence $\z{f_n}\subset L^1(I,C)$ such that $f_n\xrightharpoonup[n\to\infty]{L^1(I,E)}f$. In the context of \cite[Theorem 3.12(c)]{kunze} and \eqref{k} one may estimate 
\begin{align*}
\beta\f(\x{K(f_n)(t)}\g)&=\beta\f(\f\{U(t)g(K(f_n))+\int_0^tU(t-s)f_n(s)\,{\rm d}s\g\}_{n=1}^\infty\g)\<\beta\f(\f\{U(t)g(K(f_n))\g\}_{n=1}^\infty\g)\\&+\beta\f(\f\{\int_0^tU(t-s)f_n(s)\,{\rm d}s\g\}_{n=1}^\infty\g)\<Mk\sup_{\tau\in I}\beta\f(\{K(f_n)(\tau)\}_{n=1}^\infty\g)+4M\overline{\int_0^t}\beta\f(\x{f_n(s)}\g)\,{\rm d}s\\&\<Mk\sup_{\tau\in I}\beta\f(\f\{K(f_n)(\tau)\g\}_{n=1}^\infty\g)+4Mt\beta(C)
\end{align*}
for $t\in I$. Eventually, \[\sup_{t\in I}\beta\f(\f\{K(f_n)(t)\g\}_{n=1}^\infty\g)\<Mk\sup_{t\in I}\beta\f(\f\{K(f_n)(t)\g\}_{n=1}^\infty\g),\] which means that $\beta\f(\f\{K(f_n)(t)\g\}_{n=1}^\infty\g)=0$ for each $t\in I$.\par Fix $\tau\in I$ and $\eps>0$. There is $\delta>0$ such that $2M||C||^+\delta<\frac{\eps}{3}$. Let $s\in I$ be such that $|\tau-s|<\delta$. Since $\x{K(f_n)(s)}$ is relatively compact, the family $\x{U(\cdot)K(f_n)(s)}$ is equicontinuous i.e., \[\sup_\n|U(t-s)K(f_n)(s)-U(\tau-s)K(f_n)(s)|\xrightarrow[t\to\tau]{}0.\] Clearly there is $\delta_1\in(0,\delta)$ such that, for any $\n$,
\begin{align*}
|K(f_n)(t)-K(f_n)(\tau)|&=\f|U(t-s)K(f_n)(s)-U(\tau-s)K(f_n)(s)+\int_s^t\!U(t-z)f_n(z)\,{\rm d}z-\int_s^\tau\!U(\tau-z)f_n(z)\,{\rm d}z\g|\\&\<|U(t-s)K(f_n)(s)-U(\tau-s)K(f_n)(s)|+N||C||^+|t-s|+N||C||^+|\tau-s|<\eps
\end{align*}
for $t\in I$ with $|t-\tau|<\delta_1$. In other words, $\x{K(f_n)}$ is equicontinuous. In view of the Arzel\`a theorem, we see that (passing to a subsequence if necessary) $K(f_n)\xrightarrow[n\to\infty]{C(I,E)}y$. Notice that $U(\cdot)g(K(f_n))\xrightarrow[n\to\infty]{C(I,E)}U(\cdot)g(y)$. Since the linear part of $K(f_n)$ tends weakly to $\int_0^\cdot U(\cdot-s)f(s)\,{\rm d}s$, one sees that \[y(t)\xleftharpoonup[n\to\infty]{E}K(f_n)(t)\xrightharpoonup[n\to\infty]{E}U(t)g(y)+\int_0^tU(t-s)f(s)\,{\rm d}s,\;\;\text{for }t\in I.\] Hence, $y=K(f)$, i.e. $K\colon L^1(I,C)\to C(I,E)$ is strongly continuous.
\par We claim that $K$ has convex fibers. Fix $\lambdaup\in(0,1)$. If $f,h\in K^{-1}(\{x\})$, then $K(f)=K(h)$. From this it follows that $\int_0^tU(t-s)(f-h)(s)\,{\rm d}s=0$ for $t\in I$. Hence,
\begin{align*}
|K(\lambdaup f+(1-\lambdaup)h)(t)-K(h)(t)|&=\f|U(t)g(K(\lambdaup f+(1-\lambdaup)h))-U(t)g(K(h))+\lambdaup\int_0^tU(t-s)(f-h)(s)\,{\rm d}s\g|\\&\<M|g(K(\lambdaup f+(1-\lambdaup)h))-g(K(h))|\<ML\,||K(\lambdaup f+(1-\lambdaup)h)-K(h)||
\end{align*}
for $t\in I$. Consequently, $K(\lambdaup f+(1-\lambdaup)h)=K(h)$ i.e., $\lambdaup f+(1-\lambdaup)h\in K^{-1}(\{x\})$.\par Condition \eqref{R} is nothing more than \[\int\limits_0^T\mu(s)\,{\rm d}s<\int\limits_{M(Lr+|g(0)|)}^r\frac{\rm{d}s}{M(1+s)}.\] The latter entails existence of an $(K\circ N_F)$-invariant subset $W_0\subset C(I,E)$, which has a nonempty and convex interior (see proof of \cite[Theorem 3.3]{zhu}). Therefore, the operator $H\colon\overline{\text{Int}W_0}\map C(I,E)$ satisfies condition \eqref{LS} (cf. Remark \ref{rem}(ii)).\par Let $\Omega\subset W_0$ be such that $\overline{\Omega}\subset\overline{\co}\,\big(\{x_0\}\cup H(\Omega)\big)$ for some $x_0\in\text{Int}W_0$. Let $\w{\beta}$ be a measure of non-compactness on $E$ generated by the sequential Hausdorff MNC i.e., \[\w{\beta}(\Omega):=\max\f\{\beta(D)\colon D\subset\Omega\text{ denumerable}\g\}.\] Clearly, $\w{\beta}$ is nonsingular and monotone. Fix $t\in I$. Assume that $\w{\beta}(H(\Omega)(t))=\beta\big(\x{K(f_n)(t)}\big)$ for some $\x{K(f_n)}\subset H(\Omega)$. By virtue of \cite[Theorem 3.12]{kunze} one gets
\begin{align*}
\w{\beta}(H(\Omega)(t)&=\beta\big(\x{K(f_n)(t)}\big)=\beta\f(\f\{U(t)g(K(f_n))+\int_0^tU(t-s)f_n(s)\,{\rm d}s\g\}_{n=1}^\infty\g)\\&\<M\beta\f(g\f(\x{K(f_n)}\g)\g)+cM\int_0^t\beta\f(\x{f_n(s)}\g)\,{\rm d}s\\&\<Mk\sup_{\tau\in I}\beta\f(\x{K(f_n)(\tau)}\g)+cM\int_0^t\eta(s)\sup_{\tau\in I}\w{\beta}(\Omega(\tau))\,{\rm d}s\\&\<Mk\sup_{\tau\in I}\w{\beta}(H(\Omega)(\tau))+cM\int_0^t\eta(s)\,{\rm d}s\sup_{\tau\in I}\w{\beta}(H(\Omega)(\tau)),
\end{align*}
where $c=1$ in case $E$ has property ${\mathcal S}$ and $c=4$ otherwise. Since $t$ was chosen arbitrary, \[\sup_{t\in I}\w{\beta}(H(\Omega)(t))\<M(k+c||\eta||_1)\sup_{t\in I}\w{\beta}(H(\Omega)(t)).\] Therefore, $\sup\limits_{t\in I}\w{\beta}(\Omega(t))=0$. Clearly, $\w{\beta}^{-1}(\{0\})$ is contained in the family of relatively compact subsets of the space $E$. This means that $\Omega(t)$ is relatively compact for every $t\in I$. In other words, condition \eqref{monch3} is satisfied. \par In view of Theorem \ref{solvability} the set-valued operator $H\colon\overline{\text{Int}W_0}\map C(I,E)$ is admissible. From Theorem \ref{fixed} it follows that $H$ has at least one fixed point. This fixed point constitutes a solution of the nonlocal Cauchy problem \eqref{inclusion}. 
\end{proof}
\begin{remark}
{\em The preceding result may be treated as a refinement of \cite[Theorem 3.8]{zhu}. Specifically, we did not assume neither equicontinuity of the semigroup $\{U(t)\}_{t\geqslant 0}$ nor that $F(t,\cdot)$ must be upper semicontinuous.}
\end{remark}

\noindent{\bf Example 3.} For the third example we consider the following so-called Hammerstein integral inclusion:
\begin{equation}\label{hammerstein}
x(t)\in h(t)+\int_0^Tk(t,s)F(s,x(s))\,{\rm d}s\;\;\text{a.e. on }I
\end{equation}
with $h\in L^p(I,E)$. We shall assume the following hypotheses about the kernel mapping $k\colon I^2\to{\mathcal L}(E)$ :
\begin{itemize}
\item[$(\ka_1)$] the function $k\colon I\times I\to{\mathcal L}(E)$ is strongly measurable in a product measure space,
\item[$(\ka_2)$] for every $t\in I$, $k(t,\cdot)\in L^r(I,{\mathcal L}(E))$ with $r\in(1,\infty]$ being the conjugate exponent of $q$, i.e. $q^{-1}+r^{-1}=1$, 
\item[$(\ka_3)$] the function $I\ni t\mapsto k(t,\cdot)\in L^r(I,{\mathcal L}(E))$ belongs to $L^p(I,L^r(I,{\mathcal L}(E)))$.
\end{itemize}
\begin{theorem}\label{hamth}
Let $1\<q\<p<\infty$. Assume that $k\colon I^2\to{\mathcal L}(E)$ satisfies $(\ka_1)$--$(\ka_3)$, while $F\colon I\times E\map E$ fulfills $(\F_1)$--$(\F_4)$ together with
\begin{itemize}
\item[$(\F_6')$] for every closed separable subspace $E_0$ of $E$ there exists a function $\eta_{E_0}\in L^\frac{pq}{p-q}(I,\R{})$ such that for all bounded subsets $\Omega\subset E_0$ and for a.a. $t\in I$ the inequality holds \[\beta_{E_0}(F(t,\Omega)\cap E_0)\<\eta_{E_0}(t)\beta_{E_0}(\Omega).\]
\end{itemize}
If there is an $R>0$ such that
\begin{equation}\label{ls2}
||\,||k(t,\cdot)||_r\,||_p\f(||b||_q+c\f(R+||h||_p\g)^\frac{p}{q}\g)\<R
\end{equation}
and
\begin{equation}\label{hamcond}
||\,||k(t,\cdot)||_r\,||_p||\eta_{E_0}||_{\frac{pq}{p-q}}<1,
\end{equation}
then the integral inclusion \eqref{hammerstein} has at least one $p$-integrable solution.
\end{theorem}
\begin{proof}
Define the external operator $K\colon L^q(I,E)\to L^p(I,E)$ in the following way 
\begin{equation}\label{opK}
K(w)(t):=h(t)+\int_0^Tk(t,s)w(s)\,{\rm d}s,\;\;\;t\in I.
\end{equation}
Since, $K$ is affine, it has convex fibers. It is clear that \[||K(w_1)-K(w_2)||_p\<||\,||k(t,\cdot)||_r\,||_p||w_1-w_2||_q\] for any $w_1,w_2\in L^q(I,E)$. Thus, $K$ meets condition $(\es_2)$.\par To see that $(\es_1)$ is also satisfied, consider a compact subset $C\subset E$ and a sequence $\z{w_n}\subset L^q(I,C)$ such that $w_n\xrightharpoonup[n\to\infty]{L^q(I,E)}w$. Observe that \[\sup_\n|K(w_n)(t)|\<|h(t)|+||k(t,\cdot)||_r\sup_\n||w_n||_q\;\;\;\text{a.e. on }I,\] which means that the family $\x{K(w_n)}$ is $p$-integrably bounded. Moreover, \[\sup_\n|k(t,s)w_n(s)|\<||k(t,s)||_{\mathcal L}||C||^+\;\;\;\text{a.e. on }I\] and \[\beta\f(\x{k(t,s)w_n(s)}\g)\<||k(t,s)||_{\mathcal L}\beta\f(\x{w_n(s)}\g)=0\;\;\;\text{a.e. on }I.\] Hence, by \cite[Corollary 3.1]{heinz}, we have \[\beta\f(\x{K(w_n)(t)}\g)=\beta\f(\f\{h(t)+\int_0^Tk(t,s)w_n(s)\,{\rm d}s\g\}_{n=1}^\infty\g)\<\beta\f(\f\{\int_0^Tk(t,s)w_n(s)\,{\rm d}s\g\}_{n=1}^\infty\g)\<2\int_0^T0\,{\rm d}t=0\] for a.a. $t\in I$. Employing \cite[Corollary 3.1]{heinz} again one obtains:
\[\beta\f(\f\{\int_s^tK(w_n)(\tau)\,{\rm d}\tau\g\}_{n=1}^\infty\g)\<2\int_s^t0\,{\rm d}\tau=0\] for every $0<s<t<T$. In other words the sets $\f\{\int_s^tK(w_n)(\tau)\,{\rm d}\tau\g\}_{n=1}^\infty$ are relatively compact in $E$. On the other hand, the following estimate holds
\begin{align*}
\int_0^{T-h}|K(w_n)(t+h)-K(w_n)(t)|^p\,{\rm d}t&\<\int_0^{T-h}\f(\int_0^T||k(t+h,s)-k(t,s)||_{\mathcal L}|w_n(s)|\,{\rm d}s\g)^p\,{\rm d}t\\&\<\sup_\n||w_n||_q^p\int_0^{T-h}||k(t+h,\cdot)-k(t,\cdot)||_r^p\,{\rm d}t.
\end{align*}
Bearing in mind that the singleton set $\{t\mapsto k(t,\cdot)\}$ is compact in $L^p(I,L^r(I,{\mathcal L}(E)))$, the following convergence is self-evident \[\sup_\n||w_n||_q^p\int_0^{T-h}||k(t+h,\cdot)-k(t,\cdot)||_r^p\,{\rm d}t\xrightarrow[h\to 0^+]{}0.\] Therefore, the set $\x{K(w_n)}$ is $p$-equiintegrable. In view of \cite[Theorem 2.3.6]{papa}, $K(w_n)\xrightarrow[n\to\infty]{L^p(I,E)}y$, up to a subsequence. At the same time
\[y(t)\xleftharpoonup[n\to\infty]{E}K(w_n)(t)\xrightharpoonup[n\to\infty]{E}h(t)+\int_0^Tk(t,s)w(s)\,{\rm d}s,\;\;\text{for }t\in I.\] Eventually, $K(w_n)\xrightarrow[n\to\infty]{L^p(I,E)}K(w)$.\par Let $R>0$ be matched according to \eqref{ls2} and $H\colon D_{L^p}(h,R)\map L^p(I,E)$ be defined as usual as $H:=K\circ N_F$. Take $u\in D_{L^p}(h,R)$ and $w\in N_F(u)$. Then 
\begin{equation}\label{K(w)}
\begin{split}
|K(w)(t)|&\<|h(t)|+\int_0^T||k(t,s)||_{\mathcal L}|w(s)|\,{\rm d}s\<|h(t)|+\int_0^T||k(t,s)||_{\mathcal L}\f(b(s)+c|u(s)|^\frac{p}{q}\g)\,{\rm d}s\\&\<|h(t)|+||k(t,\cdot)||_r\f(||b||_q+c||u||_p^\frac{p}{q}\g)\<|h(t)|+||k(t,\cdot)||_r\f(||b||_q+c\f(R+||h||_p\g)^\frac{p}{q}\g),
\end{split}
\end{equation}
i.e. the range of the operator $H$ is uniformly $p$-integrable. Note that the assumption \eqref{ls2} is nothing but condition \eqref{radius1} formulated in the context of the Hammerstein inclusion \eqref{hammerstein}.\par Let $M\subset D_{L^p}(h,R)$ be a countable subset of $\co(\{h\}\cup H(\overline{M}))$. Then there is a subset $\x{v_n}\subset H(\overline{M})$ such that $M\subset\co\f(\{h\}\cup\x{v_n}\g)$. Assume that $v_n=K(w_n)$ and $w_n\in N_F(u_n)$ with $u_n\in\overline{M}$. In view of the Pettis measurability theorem there exists a closed linear separable subspace $E_0$ of $E$ such that \[\{h(t)\}\cup\x{v_n(t)}\cup\x{w_n(t)}\cup\f\{\int_0^Tk(t,s)w_n(s)\,{\rm d}s\g\}_{n=1}^\infty\cup\overline{M}(t)\subset E_0\;\;\text{for a.a. }t\in I.\] Let $\mu\in L^p(I,\R{})$ be such that $\mu(t):=|h(t)|+||k(t,\cdot)||_r\f(||b||_q+c\f(R+||h||_p\g)^\frac{p}{q}\g)$. From \eqref{K(w)} it follows that $|v_n(t)|\<\mu(t)$ a.e. on $I$. Since $\overline{M}(t)\subset\overline{\co}\f(\{h(t)\}\cup\x{v_n(t)}\g)$, we infer that $|u_n(t)|\<\mu(t)$ a.e. on $I$ for every $\n$. Eventually, \[\sup_\n|w_n(t)|\<\sup_\n||F(t,u_n(t))||^+\<b(t)+c\sup_\n|u_n(t)|^\frac{p}{q}\<b(t)+c\mu(t)^\frac{p}{q}\;\;\text{a.e. on }I,\] i.e. the family $\x{w_n}$ is $q$-integrably bounded. Observe that operator $K$ meets assumptions of \cite[Lemma 4.3]{precup}. Particularly, condition (S1) is satisfied for the kernel $\w{k}\colon I\to\R{}_+$ such that $\w{k}(t,s):=||k(t,s)||_{\mathcal L}$. Therefore, 
\begin{equation}\label{v_n}
\beta_{E_0}\f(\x{v_n(t)-h(t)}\g)=\beta_{E_0}\f(\f\{\int_0^Tk(t,s)w_n(s)\,{\rm d}s\g\}_{n=1}^\infty\g)\<\int_0^T||k(t,s)||_{\mathcal L}\beta_{E_0}\f(\x{w_n(s)}\g)\,{\rm d}s.
\end{equation}
Under assumption $(\F_6')$ the following estimate holds:
\begin{align*}
\beta_{E_0}\f(\x{w_n(s)}\g)&\<\beta_{E_0}\f(F\f(s,\x{u_n(s)}\g)\cap E_0\g)\<\eta_{E_0}(s)\beta_{E_0}\f(\x{u_n(s)}\g)\<\eta_{E_0}(s)\beta_{E_0}\f(M(s)\g)\\&\<\eta_{E_0}(s)\beta_{E_0}\f(\x{v_n(s)-h(s)}\g).
\end{align*}
Using the latter in the context of \eqref{v_n}, one gets \[\f\Arrowvert\beta_{E_0}\f(\x{v_n(\cdot)-h(\cdot)}\g)\g\Arrowvert_p\<||\,||k(t,\cdot)||_r\,||_p||\eta_{E_0}||_\frac{pq}{p-q}\f\Arrowvert\beta_{E_0}\f(\x{v_n(\cdot)-h(\cdot)}\g)\g\Arrowvert_p.\] Now, assumption \eqref{hamcond} entails $\beta_{E_0}\f(\x{v_n(t)-h(t)}\g)=0$ a.e. on $I$. This means that vertical slices $\x{v_n(t)}$ are relatively compact for a.a. $t\in I$. Consequently, $M(t)$ is relatively compact a.e. on $I$. In this regard, the following condition is met 
\begin{equation}\label{munch}
\f(M\subset\overline{U}\;\text{countable and }M\subset\co\,\f(\{x_0\}\cup H(\overline{M})\g)\g)\Longrightarrow M(t)\text{ is relatively compact for a.a. }t\in I.
\end{equation}
It may be proven in a strictly analogous way to Lemma \ref{(3)} that \eqref{munch} entails 
\[\f.\begin{gathered}
M\subset\overline{U},\;M\subset\co\,(\{x_0\}\cup H(M))\\
\text{and }\overline{M}=\overline{C}\text{ with }C\subset M\text{ countable }
\end{gathered}\g\}\Longrightarrow\overline{M}\text{ is compact}.\]
Therefore, operator $H\colon D_{L^p}(h,R)\map L^p(I,E)$ satisfies condition \eqref{chnom2}. Since $K$ is linear continuous, $H$ is a compact convex valued upper semicontinuous map. In view of Theorem \ref{fixed2} there exists a solution of the Hammerstein integral inlusion \eqref{hammerstein}, contained in $D_{L^p}(h,R)$.
\end{proof}
\begin{remark}
{\em The case $p=\infty$ is a little bit tricky as far as it comes to proving the relative compactness of $\x{K(w_n)}$. The easiest way to avoid such speculations is to impose the following assumption 
\begin{itemize}
\item[$(\ka_3')$] the function $I\ni t\mapsto k(t,\cdot)\in L^r(I,{\mathcal L}(E))$ is continuous 
\end{itemize}
and to make use of the classical Arzel\`a criterion.}
\end{remark}
\begin{remark}
{\em The above proven result is essentially \cite[Corollary 4.5]{precup} formulated in the context of the integral inclusion \eqref{hammerstein}. The issue of the existence of continuous solutions to Hammerstein integral inclusion is well established in the literature of the subject. One such result was proved by the author in \cite{pietkun}.}
\end{remark}

\noindent{\bf Example 4.} The following Volterra integral inclusion
\begin{equation}\label{volterra}
x(t)\in h(t)+\int_0^tk(t,s)F(s,x(s))\,{\rm d}s,\;\;\text{a.e. on }I
\end{equation}
is a special case of the problem \eqref{hammerstein} with $k\colon I^2\to{\mathcal L}(E)$ such that $k(t,s)=0$ for $t<s$. The subsequent existence result concerning inclusion \eqref{volterra} stems from the application of Theorem \ref{solvability}. 
\begin{theorem}
Let $h\in L^p(I,E)$. Suppose that all assumptions of Theorem \ref{hamth} are satisfied with the exception of \eqref{hamcond}. Then the integral inclusion \eqref{volterra} possesses a $p$-integrable solution.
\end{theorem}
\begin{proof}
In the Volterra case, the mapping $K\colon L^q(I,E)\to L^p(I,E)$ should by defined as follows: \[K(w)(t):=h(t)+\int_0^tk(t,s)w(s)\,{\rm d}s,\;\;\;t\in I.\] In order to demonstrate the thesis it is sufficient to give reason for condition \eqref{munch}. The proof of this property goes exactly the same as previously until one reaches the estimate \eqref{v_n}. Here, we have
\begin{align*}
\int_0^t\beta_{E_0}\f(\x{v_n(s)-h(s)}\g)^p{\rm d}s&=\int_0^t\beta_{E_0}\f(\f\{\int_0^sk(s,\tau)w_n(\tau)\,{\rm d}\tau\g\}_{n=1}^\infty\g)^p{\rm d}s\\&\<\int_0^t\f(\int_0^s||k(s,\tau)||_{\mathcal L}\beta_{E_0}\f(\x{w_n(\tau)}\g)\,{\rm d}\tau\g)^p{\rm d}s\\&\<\int_0^t\f(\int_0^s||k(s,\tau)||_{\mathcal L}\eta_{E_0}(\tau)\beta_{E_0}\f(\x{v_n(\tau)-h(\tau)}\g)\,{\rm d}\tau\g)^p{\rm d}s\\&\<\int_0^t||k(s,\cdot)||_r^p||\eta_{E_0}||_\frac{pq}{p-q}^p\int_0^s\beta_{E_0}\f(\x{v_n(\tau)-h(\tau)}\g)^p{\rm d}\tau\,{\rm d}s
\end{align*}
for every $t\in I$. Hence, $\beta_{E_0}\f(\x{v_n(t)-h(t)}\g)=0$ a.e. on $I$ by the Gronwall inequality. This shows that condition \eqref{munch} is fulfilled in the Volterra case as well.
\end{proof}
\begin{remark}
{\em Of course, modifying the assumptions about the integral kernel $k$ accordingly, it is not difficult to demonstrate the existence of continuous solutions to Volterra inclusion $($cf. \cite[Theorem 5]{pietkun}$)$.}
\end{remark}
\begin{corollary}
Let $1\<p=q<\infty$ and $h\in L^p(I,E)$. Suppose that all assumptions of Theorem \ref{hamth} are satisfied with the exception of \eqref{ls2} and \eqref{hamcond}. Then the integral inclusion \eqref{volterra} possesses a $p$-integrable solution.
\end{corollary}
\begin{proof}
Exclusion of assumption  \eqref{ls2} means for us necessity of showing that $H\colon D_{L^p}(h,R)\map L^p(I,E)$ satisfies boundary condition \eqref{LS} for some radius $R>0$. To this aim, put \[R:=2^{1-\frac{1}{p}}||\,||k(t,\cdot)||_r\,||_p\f(||b||_p+c||h||_p\g)e^{\frac{1}{p}2^{p-1}c^p||\,||k(t,\cdot)||_r\,||_p^p}.\] Assume that $u\in\partial D_{L^p}(h,R)$. If $\lambdaup(u-h)\in H(u)-h$, then
\[\lambdaup|u(t)-h(t)|\<||k(t,\cdot)||_r||b||_p+c||k(t,\cdot)||_r\f(\int_0^t\f(|u(s)-h(s)|+|h(s)|\g)^p\,{\rm d}s\g)^\frac{1}{p}\] for a.a. $t\in I$. Whence
\[\int_0^t|u(s)-h(s)|^p\,{\rm d}s\<\lambdaup^{-p}2^{p-1}||\,||k(t,\cdot)||_r\,||_p^p\f(||b||_p+c||h||_p\g)^p+\lambdaup^{-p}2^{p-1}c^p\int_0^t||k(s,\cdot)||_r^p\int_0^s|u(\tau)-h(\tau)|^p{\rm d}\tau\,{\rm d}s\] for every $t\in I$. The Gronwall inequality implies
\[\int_0^t|u(s)-h(s)|^p\,{\rm d}s\<\lambdaup^{-p}2^{p-1}||\,||k(t,\cdot)||_r\,||_p^p\f(||b||_p+c||h||_p\g)^p\exp\f(\lambdaup^{-p}2^{p-1}c^p\int_0^t||k(s,\cdot)||_r^p\,{\rm d}s\g)\] for $t\in I$. Eventually, the following estimation holds 
\begin{equation}\label{xxx}
||u-h||_p\<\frac{1}{\lambdaup}2^{1-\frac{1}{p}}||\,||k(t,\cdot)||_r\,||_p\f(||b||_p+c||h||_p\g)e^{\frac{1}{p}\lambdaup^{-p}2^{p-1}c^p||\,||k(t,\cdot)||_r\,||_p^p}.
\end{equation}
If $\lambdaup>1$ would be the case, then \eqref{xxx} leads to \[||u-h||_p<2^{1-\frac{1}{p}}||\,||k(t,\cdot)||_r\,||_p\f(||b||_p+c||h||_p\g)e^{\frac{1}{p}2^{p-1}c^p||\,||k(t,\cdot)||_r\,||_p^p},\] which is in contradiction with the definition of radius $R$. Therefore $\lambdaup\<1$ and the Yamamuro's condition is satisfied (Remark \ref{rem}).
\end{proof}

\noindent{\bf Example 5.}
Let $\f(\mathbb{H},|\cdot|_{\mathbb H}\g)$ be a separable Hilbert space and $(E,|\cdot|_E)$ be a subspace of $\mathbb{H}$ carrying the structure of a separable reflexive Banach space, which embeds into $\mathbb{H}$ continuously and densly. Identifying $\mathbb{H}$ with its dual we obtain $E\hookrightarrow \mathbb{H}\hookrightarrow E^*$, with all embeddings being continuous and dense. Such a triple of spaces is usually called evolution triple (\cite[p. 416]{zeidler}). Let us note that since $E\hookrightarrow \mathbb{H}\hookrightarrow E^*$ continuously, there exist constants $L,M>0$ such that $|\cdot|_\mathbb{H}\<L|\cdot|_E$ and $|\cdot|_{E^*}\<M|\cdot|_\mathbb{H}$. Without any loss of generality and to simplify our calculations we may take $L=M=1$.\par In this example we will investigate an abstract evolution inclusion of the following form:
\begin{equation}\label{inclus}
\begin{gathered}
\dot{x}(t)+A(t,x(t))+F(t,x(t))\ni h(t),\;\;\text{a.e. on }I:=[0,1],\\
x(0)=x(1),
\end{gathered}
\end{equation}
where $A\colon I\times E\map E^*$, $F\colon I\times \mathbb{H}\map \mathbb{H}$ and $h\in L^q(I,E^*)$. \par Let $2\<p<\infty$ and $1<q\<2$ be H\"older conjugates. We assume an agreement regarding the following notations: $\e:=L^p(I,E)$, $\e^*:=L^q(I,E^*)$, $\h:=L^p(I,\mathbb{H})$, $\h^*:=L^q(I,\mathbb{H})$, $\langle\langle\cdot,\cdot\rangle\rangle:=\langle\cdot,\cdot\rangle_{\e^*\times\e}$ $($the duality brackets for the pair $(\e^*,\e))$.\par The symbol $W_p^1$ stands for the Bochner-Sobolev space \[W_p^1(0,1;E,\mathbb{H}):=\left\{x\in\e\colon\dot{x}\in\e^*\right\},\] where the derivative $\dot{x}$ is understood in the sense of vectorial distributions. Denote by $N_A\colon\e\map\e^*$ the Nemytski\v{\i} operator corresponding to the multimap $A$. By a solution of \eqref{inclus} we mean a function $x\in W_p^1$ such that $\dot{x}(t)+g(t)+f(t)=h(t)$ a.e. on $I$, $x(0)=x(1)$ with $g\in N_A(x)$ and $f\in N_F(x)$.\par Our hypothesis on the operator $A\colon I\times E\map E^*$ are as follows:
\begin{itemize}
\item[$(\A_1)$] for every $(t,x)\in I\times E$ the set $A(t,x)$ is nonempty closed and convex,
\item[$(\A_2)$] the map $t\mapsto A(t,x)$ has a measurable selection for every $x\in E$,
\item[$(\A_3)$] for a.a. $t\in I$, the operator $A(t,\cdot)\colon E\map E^*$ is hemicontinuous (i.e. $\lambdaup\mapsto A(t,x+\lambdaup y)$ is usc from $[0,1]$ into $(E^*,w)$ for all $x,y\in E$),
\item[$(\A_4)$] the map $x\mapsto A(t,x)$ is monotone,
\item[$(\A_5)$] there exists a nonnegative function $a\in L^q(I,\R{})$ and a constant $\hat{c}>0$ such that for all $x\in E$ and for a.a. $t\in I$,\[||A(t,x)||^+_{E^*}\<a(t)+\hat{c}|x|_E^{\frac{p}{q}},\]
\item[$(\A_6)$] there exists a constant $d>0$ such that for all $x\in E$ and a.e. on $I$, \[d|x|_E^p\<\langle A(t,x),x\rangle^-:=\inf\{\langle y,x\rangle\colon y\in A(t,x)\}.\]
\end{itemize}

\begin{theorem}\label{th1}
Let $(E,\mathbb{H},E^*)$ be an evolution triple such that $E\hookrightarrow \mathbb{H}$ compactly. Assume that conditions $(\A_1)$--$(\A_6)$ and $(\F_1)$--$(\F_4)$ are satisfied. Suppose further that the following inequality holds 
\begin{equation}\label{stala}
c<d.
\end{equation}
Then problem \eqref{inclus} has at least one solution. Moreover, these solutions form a compact subset of the space $(\h,||\cdot||_\h)$.
\end{theorem}

\begin{proof}
Let $A_h\colon I\times E\map E^*$ be defined by $A_h(t,x):=A(t,x)-h(t)$. Evidently, the multimap $A_h$ meets conditions $(\A_1)$-$(\A_6)$, with the proviso that \[||A_h(t,x)||^+_{E^*}\<(a(t)+|h(t)|_{E^*})+\hat{c}|x|_E^\frac{p}{q}\;\;\text{and}\;\;\langle A_h(t,x),x\rangle^-\geqslant d|x|^p_E-|h(t)|_{E^*}|x|_E.\] Observe that conditions $(\F_1)$-$(\F_4)$ are fulfilled by the multimap $A_h$. Indeed, assumptions $(\A_1)$, $(\A_3)$,~$(\A_4)$ entail the strong-weak sequential closedness of the graph of $A_h$ while weak upper semicontinuity of the map $x\mapsto A_h(t,x)$ follows from assumptions $(\A_3)$ and $(\A_4)$ (see \cite[Proposition 3.2.18]{papa}). Therefore, the Nemytski\v{\i} operator $N_{A_h}\colon\e\map\e^*$ is a convex weakly compact valued weakly upper semicontinuous map (cf. Corollary \ref{niemycor}). As such it is maximal monotone (cf. \cite[Proposition 3.2.19]{papa}).\par Denote by the symbol $L\colon D(L)\to\e^*$ a continuous linear differential operator $L:=\frac{d}{dt}$, with the domain \[D(L):=\left\{x\in W_p^1\colon x(0)=x(1)\right\}\] being a closed linear subspace of $W_p^1$. In view of \cite[Proposition 32.10]{zeidler} this operator is maximal monotone as well. Therefore, the sum $L+N_{A_h}\colon D(L)\map\e^*$ must be maximal monotone (by \cite[Theorem 3.2.41]{papa}). Since \[\frac{\langle\langle Lx,x\rangle\rangle+\langle\langle N_{A_h}(x),x\rangle\rangle^-}{||x||_\e}\geqslant\frac{d||x||_\e^p-||h||_{\e^*}||x||_\e}{||x||_\e}=d||x||^{p-1}_\e-||h||_{\e^*}\xrightarrow[||x||_\e\to+\infty]{}+\infty,\] the map $L+N_{A_h}$ is coercive and the equality $(L+N_{A_h})(D(L))=\e^*$ follows (cf. \cite[Cor.3.2.31]{papa}). It is an immediate consequence of definition that the set-valued inverse operator $(L+N_{A_h})^{-1}\colon\e^*\map\e$ is also maximal monotone, i.e. \[\langle\langle x_1-y_1,x_2-y_2\rangle\rangle\geqslant 0\;\;\forall\,(x_1,x_2)\in\Graph((L+N_{A_h})^{-1})\Longrightarrow (y_1,y_2)\in\Graph((L+N_{A_h})^{-1})\subset\e^*\times D(L).\] \par It is easy to realize that \eqref{inclus} is in fact an inclusion of Hammerstein type. Since \[Lx\in-N_{A_h}(x)+N_{(-F)}(x)\Leftrightarrow x\in(L+N_{A_h})^{-1}\!\!\circ N_{(-F)}(x),\] it is fully understandable that the external operator $K\colon\h^*\map\h$ should be defined as $K:=(L+N_{A_h})^{-1}$. If we set $H\colon\h\map\h$ to be $H:=K\circ N_{(-F)}$, then the value of $H$ at point $u$ is a solution set of the periodic problem
\begin{equation}
\begin{gathered}
\dot{x}(t)+A(t,x(t))+F(t,u(t))\ni h(t),\;\;\text{a.e. on }I\\
x(0)=x(1).
\end{gathered}
\end{equation}
\par Notice that $K^{-1}(\{x\})\cap N_{(-F)}(u)=\f(L+N_{A_h}\g)(x)\cap N_{(-F)}(u)$ for any $x\in H(u)$. This means that operator $K$ possesses convex fibers as a map from $N_{(-F)}(u)$ onto the image $H(u)$, which is exactly what we need to be able to apply Theorem \ref{solvability}.\par Let $w_n\rightharpoonup w_0$ in $\h^*$ and $x_n\in K(w_n)$ for $\n$. This means that $x_n\in D(L)$ and $Lx_n+z_n=w_n$ for some $z_n\in N_{A_h}(x_n)$. Since $\{w_n\}_\n$ is bounded in $\h^*$, we see that 
\[||Lx_n||_{\e^*}\<||z_n||_{\e^*}+||w_n||_{\e^*}\<||N_{A_h}(x_n)||^+_{\e^*}+||w_n||_{\h^*}\<||a||_q+||h||_{\e^*}+\hat{c}||x_n||_\e^{p-1}+\sup_\n||w_n||_{\h^*}\] and \[d||x_n||^p_\e-||h||_{\e^*}||x_n||_\e\<\langle\langle z_n,x_n\rangle\rangle+\langle\langle Lx_n,x_n\rangle\rangle=\langle\langle w_n,x_n\rangle\rangle\<||w_n||_{\e^*}||x_n||_\e\<\sup_\n||w_n||_{\h^*}||x_n||_\e,\] i.e. \[||x_n||^{p-1}_\e\<d^{-1}\f(\sup_\n||w_n||_{\h^*}+||h||_{\e^*}\g).\] Thus, the sequence $\z{x_n}$ is bounded in $W_p^1$ and we may assume, passing to a subsequence if necessary, that $x_n\xrightharpoonup[n\to\infty]{W_p^1}x_0$. Moreover, there is a subsequence (again denoted by) $\z{z_n}$ such that $z_n\xrightharpoonup[n\to\infty]{\e^*}z_0$. Since $W_p^1$ embeds into $\h$ compactly (see \cite[Theorem 2.2.30]{papa}), we infer that $(x_n)_\n$ tends strongly, up to a subsequence, to $x_0$ in the norm of $\h$. Observe that $\langle\langle\dot{x}_n,x_n-x_0\rangle\rangle=\langle\langle\dot{x}_0,x_n-x_0\rangle\rangle$ and so
\begin{align*}
\limsup_{n\to\infty}\,\langle\langle z_n,x_n-x_0\rangle\rangle&\<\limsup\limits_{n\to\infty}\,\langle\langle-\dot{x}_n,x_n-x_0\rangle\rangle+\limsup_{n\to\infty}\,\langle\langle w_n,x_n-x_0\rangle\rangle\\&=\lim_{n\to\infty}\,\langle\langle-\dot{x}_0,x_n-x_0\rangle\rangle+\limsup_{n\to\infty}\,\langle w_n,x_n-x_0\rangle_{\h^*\times\h}\\&\<0+\limsup_{n\to\infty}||w_n||_{\h^*}||x_n-x_0||_\h\\&\<\sup_\n||w_n||_{\h^*}\lim_{n\to\infty}\,||x_n-x_0||_\h=0.
\end{align*}
Hence,
\begin{equation}\label{eM6}
\limsup\limits_{n\to\infty}\,\langle\langle z_n,x_n\rangle\rangle\<\langle\langle z_0,x_0\rangle\rangle.
\end{equation} 
Recalling the decisive monotonicity trick one easily sees that $z_0\in N_{A_h}(x_0)$. Indeed, by \eqref{eM6} we have
\begin{equation}\label{monotrick}
\begin{split}
\langle\langle z_0-z,x_0-x\rangle\rangle&=\langle\langle z_0,x_0\rangle\rangle-\langle\langle z,x_0\rangle\rangle-\langle\langle z_0-z,x\rangle\rangle\\&\geqslant\limsup\limits_{n\to\infty}\,\f(\langle\langle z_n,x_n\rangle\rangle-\langle\langle z,x_n\rangle\rangle-\langle\langle z_n-z,x\rangle\rangle\g)=\limsup\limits_{n\to\infty}\,\langle\langle z_n-z,x_n-x\rangle\rangle\geqslant 0
\end{split}
\end{equation}
for every $(x,z)\in\Graph(N_{A_h})$. Since $N_{A_h}$ is maximal monotone, we get $z_0\in N_{A_h}(x_0)$. The weak convergence $x_n\xrightharpoonup[n\to\infty]{W_p^1}x_0$ entails $x_n\xrightharpoonup[n\to\infty]{C(I,\mathbb{H})}x_0$. Thus, $x_n(t)\xrightharpoonup[n\to\infty]{\mathbb{H}}x_0(t)$ for every $t\in I$. This means that the boundary condition $x_0(0)=x_0(1)$ is satisfied. Since \[Lx_0+z_0\xleftharpoonup[n\to\infty]{\e^*}Lx_n+z_n=w_n\xrightharpoonup[n\to\infty]{\e^*}w_0,\] it follows that $x_0\in K(w_0)$. This justifies the claim that the external operator $K\colon\h^*\map\h$ is a strongly upper semicontinuous map with compact values. Further, maximal monotonicity of $K$ implies also convexity of its values (cf. \cite[Proposition 3.2.7]{papa}). In other words, operator $K$ satisfies assumption $(\es_3)$. \par It is easy to see that operator $H$ maps bounded sets into relatively compact ones. Indeed, let $x\in H(u)$ and $w\in N_{(-F)}(u)$ be such that $x\in K(w)$. Then
\begin{equation}\label{szacu5}
\begin{split}
||Lx||_{\e^*}&\<||-N_{A_h}(x)+w||^+_{\e^*}\<||N_{A_h}(x)||^+_{\e^*}+||w||_{\e^*}\<||N_{A_h}(x)||^+_{\e^*}+||w||_{\h^*}\\&\<||a||_q+||h||_{\e^*}+\hat{c}||x||_\e^{p-1}+||b||_q+c||u||_{\h}^{p-1}.
\end{split}
\end{equation}
and 
\[d||x||^p_\e-||h||_{\e^*}||x||_\e\<\langle\langle N_{A_h}(x),x\rangle\rangle^-\<\langle\langle w,x\rangle\rangle-\langle\langle Lx,x\rangle\rangle\<||w||_{\e^*}||x||_\e-0\<\f(||b||_q+c||u||_{\h}^{p-1}\g)||x||_\e.\] If $u\in D_\h(0,R)$, then
\begin{equation}\label{inv2}
||x||^{p-1}_\e\<d^{-1}\f(||b||_q+cR^{p-1}+||h||_{\e^*}\g).
\end{equation}
Taking into account \eqref{szacu5}, we arrive at the estimate
\[||Lx||_{\e^*}\<||a||_q+||h||_{\e^*}+||b||_q+(1+d^{-1}\hat{c})cR^{p-1}+d^{-1}\hat{c}\f(||b||_q+||h||_{\e^*}\g).\] Therefore, the image $H(D_\h(0,R))$ is bounded in $W_p^1$. Since the embedding $W_p^1\hookrightarrow\h$ is compact, the set $H(D_\h(0,R))$ must have a compact closure in the space $\h$. This means in particular that operator $H\colon D_\h(0,R)\map\h$ satisfies condition \eqref{monch}.
\par In order to complete the proof we will choose a radius $R>0$ in such a way that $u\not\in\lambdaup H(u)$ on $\partial D_\h(0,R)$ for all $\lambdaup\in(0,1)$. Let \[R:=\f(\frac{||b||_q+||h||_{\e^*}}{d-c}\g)^\frac{q}{p}.\] This definition is correct, since we have assumed \eqref{stala}. Suppose that $u\in\partial D_\h(0,R)$ and $\lambdaup u\in H(u)$, i.e. $L(\lambdaup u)+w\in-N_{A_h}(\lambdaup u)$ for some $w\in N_F(u)$. One easily sees that 
\begin{equation}\label{equ10}
\begin{split}
d||\lambdaup u||_\e^p-||h||_{\e^*}||\lambdaup u||_\e&\<\langle\langle N_{A_h}(\lambdaup u),\lambdaup u\rangle\rangle^-\<\langle\langle-L(\lambdaup u)-w,\lambdaup u\rangle\rangle\<||N_F(u)||_{\e^*}^+||\lambdaup u||_\e\\&\<\f(||b||_q+c||u||_{\h}^{p-1}\g)||\lambdaup u||_\e
\end{split}
\end{equation}
and so \[d||u||_{\h}^{p-1}\<d||u||_{\e}^{p-1}\<\lambdaup^{1-p}\f(||b||_q+c||u||_{\h}^{p-1}+||h||_{\e^*}\g).\] The latter implies that $\lambdaup\<1$. Otherwise, the following inequality would have to be satisfied \[R<\f(\frac{||b||_q+||h||_{\e^*}}{d-c}\g)^\frac{q}{p},\] which contradicts the definition of radius $R$. Therefore, operator $H\colon D_\h(0,R)\map\h$ satisfies the boundary condition \eqref{LS}.\par By virtue of Theorem \ref{solvability} operator $H$ possesses at least one fixed point. Of course, this fixed point constitutes a solution of the periodic problem \eqref{inclus}.\par The preceding estimation ensures also that $\fix(H)\subset D_\h(0,R)$. Since $H$ is completely continuous, the fixed point set $\fix(H)$ is closed and the image of this set $H(\fix(H))$ possesses a compact closure. Given that $\fix(H)\subset\overline{H(\fix(H))}$, the solution set of the periodic problem \eqref{inclus} must be compact as a subset of the space $\h$.
\end{proof}

\begin{corollary}
Under assumptions of Theorem \ref{th1}. the solution set of the periodic problem \eqref{inclus} is nonempty and compact in the norm topology of the space $C(I,\mathbb{H})$.
\end{corollary}
\begin{proof}
Suppose that $u_n\in H(u_n)$. Since the fixed point set $\fix(H)$ is bounded in $W_p^1$ (as we have actually proved that previously), we may assume, passing to a subsequence if necessary, that $u_n\xrightharpoonup[n\to\infty]{W_p^1}u$. Let $(u_{k_n})_\n$ be a subsequence strongly convergent to $u$ in $\h$. We know already that the limit point $u$ belongs to $\fix(H)$. Thus, it is sufficient to show that $\sup_{t\in I}|u_{k_n}(t)-u(t)|_\mathbb{H}\to 0$ as $n\to\infty$.\par Let $w\in N_F(u)$ be such that $\dot{u}(t)+w(t)-h(t)\in-A(t,u(t))$ a.e. on $I$. Clearly, there are $w_{k_n}\in N_F(u_{k_n})$ fulfilling $w_{k_n}\xrightharpoonup[n\to\infty]{\h^*}w$ and $\dot{u}_{k_n}(t)+w_{k_n}(t)-h(t)\in-A(t,u_{k_n}(t))$ for a.a $t\in I$. Whence \[\langle\dot{u}_{k_n}(t)+w_{k_n}(t)-\dot{u}(t)-w(t),u_{k_n}(t)-u(t)\rangle\<0,\] by $(\A_4)$. Accordingly, \[\langle\dot{u}_{k_n}(t)-\dot{u}(t),u_{k_n}(t)-u(t)\rangle\<\langle w(t)-w_{k_n}(t),u_{k_n}(t)-u(t)\rangle\] for a.a. $t\in I$ and for every $\n$.\par Observe that $u_{k_n}(t)\xrightarrow[n\to\infty]{\mathbb{H}}u(t)$ a.e. on $I$, at least for a subsequence. Fix $t_0\in I$ such that $u_{k_n}(t_0)\xrightarrow[]{\mathbb{H}}u(t_0)$. Engaging the integration by parts formula in $W_p^1$ we get
\begin{align*}
\frac{1}{2}\f(|u_{k_n}(1)-u(1)|_\mathbb{H}^2\g.&-\f|u_{k_n}(t_0)-u(t_0)|_\mathbb{H}^2\g)=\int_{t_0}^1\langle\dot{u}_{k_n}(s)-\dot{u}(s),u_{k_n}(s)-u(s)\rangle\,{\rm d}s\\&\<\int_{t_0}^1\langle w(s)-w_{k_n}(s),u_{k_n}(s)-u(s)\rangle\,{\rm d}s=\int_{t_0}^1\langle w(s)-w_{k_n}(s),u_{k_n}(s)-u(s)\rangle_\mathbb{H}\,{\rm d}s\\&\<\int_0^1|w(s)-w_{k_n}(s)|_\mathbb{H}|u_{k_n}(s)-u(s)|_\mathbb{H}\,{\rm d}s\<\f(\sup_\n||w_{k_n}-w||_{\h^*}\g)||u_{k_n}-u||_\h.
\end{align*} 
Thus \[|u_{k_n}(0)-u(0)|_\mathbb{H}^2=|u_{k_n}(1)-u(1)|_\mathbb{H}^2\xrightarrow[n\to\infty]{}0.\] Now, it is clear that \[0\<\lim_{n\to\infty}\sup_{t\in I}|u_{k_n}(t)-u(t)|_\mathbb{H}^2\<2\f(\sup_\n||w_{k_n}-w||_{\h^*}\g)\lim_{n\to\infty}||u_{k_n}-u||_\h+\lim_{n\to\infty}|u_{k_n}(0)-u(0)|_\mathbb{H}^2=0\] and so $(u_{k_n})_\n$ is norm convergent in $C(I,\mathbb{H})$.
\end{proof}

In many actual parabolic problems the presence of nonmonotone terms depending on lower-order derivatives forces us to think over the situation where the right-hand side $F$ is defined only on $I\times E$ and not on $I\times{\mathbb H}$. Application of fixed point approach in this context in conjunction with lightweight competence leads to over-restrictive assumptions, as illustrated by the following:

\begin{theorem}
Let $(E,\mathbb{H},E^*)$ be an evolution triple such that $E\hookrightarrow \mathbb{H}$ compactly. Assume that the multifunction $F\colon I\times E\map\mathbb{H}$ is such that
\begin{itemize}
\item[$($i$)$] for every $(t,x)\in I\times E$ the set $F(t,x)$ is nonempty and convex,
\item[$($ii$)$] the map $F(\cdot,x)$ has a strongly measurable selection for every $x\in E$,
\item[$($iii$)$] the graph $\Graph(F(t,\cdot))$ is sequentially closed in $(E,w)\times(\mathbb{H},w)$ for a.a. $t\in I$,
\item[$($iv$)$] there is a function $b\in L^q(I,\R{})$ and $c>0$ such that for all $x\in E$ and for a.a. $t\in I$,\[||F(t,x)||_{\mathbb{H}}^+\<b(t)+c|x|_E^{\frac{p}{q}}.\]
\end{itemize}
If hypotheses $(\A_1)$--$(\A_6)$ hold and $h\in\e^*$ is such that 
\begin{equation}\label{invariant}
\begin{cases}
c+||b||_q+||h||_{\e^*}\<d,\\
c+\hat{c}+||a||_q+||b||_q+||h||_{\e^*}\<1,
\end{cases}
\end{equation}
then problem \eqref{inclus} has a solution.
\end{theorem}
\begin{proof}
Let us use symbols $K\colon\h^*\map W_p^1$ and $N_{(-F)}\colon W_p^1\map\h^*$ to denote the operators we have defined previously in the proof of Theorem \ref{th1}. To tackle the problem of finding solutions of evolution inclusion \eqref{inclus} we shall indicate a fixed point of the superposition $H:=K\circ N_{(-F)}$.\par Assume that $u_n\xrightharpoonup[n\to\infty]{W_p^1}u$ and $x_n\in H(u_n)$ for $\n$. Let $w_n\in\h^*$ be such that $w_n\in N_{(-F)}(u_n)$ and $x_n\in K(w_n)$. Put $J:=\f\{t\in I\colon\sup_\n|u_n(t)|_E=+\infty\g\}$. Since $E$ is reflexive and the sequence $\z{u_n}$ is relatively weakly compact in the Bochner space $L^1(I,E)$, it must be uniformly integrable in view of Dunford-Pettis theorem (\cite[Theorem 2.3.24]{papa}). If $t\in J$, then for every $\lambdaup>0$ there is $n_0\in\mathbb{N}$ such that $|u_{n_0}(t)|_E\geqslant\lambdaup$. Fix $\lambdaup>0$. The following estimation can be easily verified \[\lambdaup\,\ell(J)=\int\limits_{\textstyle J}\lambdaup\,{\rm d}t\<\int\limits_{{\textstyle\{|u_{n_0}|_E\geqslant\lambdaup\}}}\!\!\!\!\!\!\lambdaup\,{\rm d}t\<\int\limits_{{\textstyle \{|u_{n_0}|_E\geqslant\lambdaup\}}}\!\!\!\!\!\!|u_{n_0}(t)|_E\,{\rm d}t\<\sup_\n\int\limits_{{\textstyle \{|u_n|_E\geqslant\lambdaup\}}}\!\!\!\!\!\!|u_n(t)|_E\,{\rm d}t.\] Now, if $\ell(J)>0$, then \[\lim_{\lambdaup\to+\infty}\sup_\n\int\limits_{{\textstyle \{|u_n|_E\geqslant\lambdaup\}}}\!\!\!\!\!\!|u_n(t)|_E\,{\rm d}t=\lim_{\lambdaup\to+\infty}\lambdaup\,\ell(J)=+\infty,\] which contradicts the uniform integrability of the sequence $\z{u_n}$. Therefore, \[\ell\f(\f\{t\in I\colon\sup_\n|u_n(t)|_E<+\infty\g\}\g)=\ell(I).\] \par Since the embedding $W_p^1\hookrightarrow\h$ is compact, we may assume that $u_n\xrightarrow[n\to\infty]{\h}u$. Thus, there exists a subset $I_0$ of full measure in $I$ such that $u_{k_n}(t)\xrightarrow[n\to\infty]{H}u(t)$ and at the same time $\sup_\n|u_{k_n}(t)|_E<+\infty$ for $t\in I_0$. Fix $t_0\in I_0$. For every bounded subsequence $\z{u_{l_{k_n}}(t_0)})$ there is a weakly convergent in $E$ subsequence $\z{u_{m_{l_{k_n}}}(t_0)}$. Clearly, $u_{m_{l_{k_n}}}(t_0)\xrightharpoonup[n\to\infty]{E}u(t_0)$ and eventually $u_{k_n}(t_0)\xrightharpoonup[n\to\infty]{E}u(t_0)$. Therefore, $u_{k_n}(t)\xrightharpoonup[n\to\infty]{E}u(t)$ a.e. on $I$.\par Using reflexivity of the space $\mathbb{H}$ and the fact that $F$ has sublinear growth and $x\mapsto F(t,x)$ is sequentially closed in $(E,w)\times(\mathbb{H},w)$ one can easily show that given a sequence $(x_n,y_n)$ in the graph $\Graph(F(t,\cdot)$ with $x_n\xrightharpoonup[n\to\infty]{E}x$, there is a subsequence $y_{k_n}\xrightharpoonup[n\to\infty]{{\mathbb H}}y\in F(t,x)$. This means, in particular, that $F(t,\cdot)$ is weakly sequentially upper hemi-continuous multimap for a.a. $t\in I$.\par Observe that $||w_{k_n}||_{\h^*}\<||b||_q+c\sup_\n||u_{k_n}||_\e^{p-1}$ for every $\n$. Of course, there is a subsequence (again denoted by) $\z{w_{k_n}}$ such that $-w_{k_n}\xrightharpoonup[n\to\infty]{\h^*}-w$. Hence, by the Convergence Theorem (cf. \cite[Lemmma 1]{cichon}), $w(t)\in -F(t,u(t))$ a.e. on $I$.\par As we have already shown in the proof of Theorem \ref{th1}, there must be a subsequence $\z{x_{l_{k_n}}}$ of the sequence $\z{x_{k_n}}$ such that $x_{l_{k_n}}\xrightharpoonup[n\to\infty]{W_p^1}x\in K(w)$. Since $w\in N_{(-F)}(u)$, it follows that $x\in H(u)$. The conducted reasoning proves in essence that for every relatively weakly compact $C\subset W_p^1$, the multimap map $H\colon(C,w)\map(W_p^1,w)$ is compact valued upper semicontinuous.\par Fix $u\in W_p^1$. The subset $N_{(-F)}(u)$ furnished with the relative weak topology of $\h^*$ is compact. Moreover, $(N_{(-F)}(u),w)$ is an acyclic space. The multimap $K\colon(N_{(-F)}(u),w)\map(H(u),w)$ may be regarded as an acyclic operator between compact topological space $N_{(-F)}(u)$ and a paracompact space $H(u)$ endowed with the relative weak topology of $W_p^1$. The fibers of this map are formed by intersections $(L+N_{A_h})(x)\cap N_{(-F)}(u)$. Hence, they are convex. Lemma \ref{acyc2} implies that the reduced Alexander--Spanier cohomologies $\w{H}^\ast((H(u),w))$ are isomorphic to $\w{H}^\ast((N_{(-F)}(u),w))$. In other words the set-valued map $H\colon(W_p^1,w)\map(W_p^1,w)$ possesses acyclic values.\par Assumption \eqref{invariant} entails \[\frac{||b||_q+||h||_{\e^*}}{d-c}\<1\;\;\text{and}\;\;\frac{||a||_q+||b||_q+||h||_{\e^*}}{1-(c+\hat{c})}\<1.\] Choose a radius $R>0$ in such a way that \[R\in\f[\max\f\{\f(\frac{||b||_q+||h||_{\e^*}}{d-c}\g)^\frac{q}{p},\frac{||a||_q+||b||_q+||h||_{\e^*}}{1-(c+\hat{c})}\g\},1\g].\] Then 
\begin{equation}\label{inv3}
d^{-1}\f(||b||_q+cR^{p-1}+||h||_{\e^*}\g)\<R^{p-1}
\end{equation}
and
\begin{equation}\label{inv4}
||a||_q+||b||_q+||h||_{\e^*}+(c+\hat{c})R^{p-1}\<R.
\end{equation}
Take an $x\in H\f(D_{W_p^1}(0,R)\g)$. Then $||x||_\e^{p-1}\<R^{p-1}$, by \eqref{inv2} and \eqref{inv3}. The latter combined with \eqref{szacu5} and \eqref{inv4} yields $||Lx||_{\e^*}\<R$. Therefore the length of vector $x$, measured in the equivalent norm $\max\{||\cdot||_\e,||L(\cdot)||_{\e^*}\}$ of the space $W_p^1$, is not greater then $R$. In other words, $H\f(D_{W_p^1}(0,R)\g)\subset D_{W_p^1}(0,R)$.\par As we have seen above, the operator $H\colon(D_{W_p^1}(0,R),w)\map(D_{W_p^1}(0,R),w)$ is an admissible multimap. Since $H\f(D_{W_p^1}(0,R)\g)$ is norm bounded in $W_p^1$, it is also a compact map. In view of the Dugundji extension theorem \cite[Theorem 4.1]{dugundji}, the ball $D_{W_p^1}(0,R)$ constitutes an absolute extensor for the class of metrizable spaces as a convex subset of a locally convex linear space $(W_p^1,w)$. Since $W_p^1$ is separable, the space $(D_{W_p^1}(0,R),w)$ must be metrizable (\cite[Proposition 3.107]{fabian}). From Theorem \ref{Lefschetz} it follows directly that the multimap $H\colon(D_{W_p^1}(0,R),w)\map(D_{W_p^1}(0,R),w)$ must have at least one fixed point.
\end{proof}

\begin{remark}
{\em Generic approach to evolution inclusions governed by operators monotone in the sense of Minty--Browder essentially comes down to the observation that the sum of densely defined monotone differential operator and suitably regular, multivalued perturbation is surjective. Proof of Theorem \ref{th1} illustrates the thesis that the ``fixed point method'' provides a real alternative to this approach.}
\end{remark}

\noindent{\bf Example 6.}
The last two examples are dedicated to boundary value problems in which the nonlinear part F possesses not necessarily convex values. The first of these concerns solvability of the Hammerstein integral inclusion \eqref{hammerstein}.
\begin{theorem}\label{nonconvex}
Let $E$ be a separable Banach space and $1\<q\<p<\infty$. Assume that $k\colon I^2\to{\mathcal L}(E)$ satisfies $(\ka_1)$--$(\ka_3)$, while $F\colon I\times E\map E$ satisfies the following conditions:
\begin{itemize}
\item[$($i$)$] for every $(t,x)\in I\times E$ the set $F(t,x)$ is nonempty and closed,
\item[$($ii$)$] the map $F(t,\cdot)$ is lower semicontinuous for each fixed $t\in I$,
\item[$($iii$)$] the map $F(\cdot,\cdot)$ is measurable with respect to the product of Lebesgue and Borel $\sigma$-fields defined on $I$ and $E$, respectively, 
\item[$($iv$)$] there is a function $b\in L^q(I,\R{})$ and $c>0$ such that for all $x\in E$ and for a.a. $t\in I$,\[||F(t,x)||^+:=\sup\{|y|_E\colon y\in F(t,x)\}\<b(t)+c|x|^{\frac{p}{q}},\]
\item[$($v$)$] there exists a function $\eta\in L^\frac{pq}{p-q}(I,\R{})$ such that for all bounded subsets $\Omega\subset E$ and for a.a. $t\in I$ the inequality holds \[\beta(F(t,\Omega))\<\eta(t)\beta(\Omega).\]
\end{itemize}
If there is an $R>0$ such that \eqref{ls2} holds together with 
\begin{equation}
||\,||k(t,\cdot)||_r\,||_p||\eta||_{\frac{pq}{p-q}}<1,
\end{equation}
then the integral inclusion \eqref{hammerstein} has at least one $p$-integrable solution.
\end{theorem}

\begin{proof}
Denote by $S(I,E)$ the space of all Lebesgue measurable functions mapping $I$ to $E$, equiped with the topology of convergence in measure. Under these assumptions (cf. \cite[p.\mbox{ }731]{ponosov}), the Nemytski\v{\i} operator $N_F\colon S(I,E)\map S(I,E)$ is lower semicontinuous. Consider a sequence $(u_n)_\n$ such that $u_n\xrightarrow[n\to\infty]{L^p(I,E)}u$. In particular, $u_n\xrightarrow[n\to\infty]{}u$ in measure. If $w$ is an arbitrary element of $N_F(u)$, then there exists a sequence $w_n\in N_F(u_n)$ such that $w_n\xrightarrow[n\to\infty]{}w$ in $S(I,E)$. From every subsequence of $(w_n)_\n$ we can extract some subsequence $(w_{k_n})_\n$ satisfying $w_{k_n}(t)\xrightarrow[n\to\infty]{E}w(t)$ a.e. on $I$. W.l.o.g we may assume that $u_{k_n}(t)\xrightarrow[n\to\infty]{E}u(t)$ for a.a. $t\in I$. Assumption (iv) means that $|w_{k_n}(t)|\<b(t)+c|u_{k_n}(t)|^\frac{p}{q}$ a.e. on $I$. Since $\x{|u_{k_n}(\cdot)|^p}$ is uniformly integrable, the latter implies that the family $\x{|w_{k_n}(\cdot)|^q}$ is uniformly integrable. Secondly, under passage to the limit one sees that $|w(t)|\<b(t)+c|u(t)|^\frac{p}{q}$ for a.a. $t\in I$, i.e. $w\in L^q(I,E)$. Hence, the uniform integrability of $\x{|w_{k_n}(\cdot)-w(\cdot)|^q}$ follows. In view of Vitali convergence theorem $\lim\limits_{n\to\infty}\int_I|w_{k_n}(t)-w(t)|^qdt=0$. Consequently, $w_n\xrightarrow[n\to\infty]{L^q(I,E)}w$. Therefore, the multimap $N_F\colon L^p(I,E)\map L^q(I,E)$ meets the definition of lower semicontinuity. It is also clear that this operator possesses closed and decomposable values. In view of \cite[Theorem 3]{brescol} there exists a continuous map $f\colon L^p(I,E)\to L^q(I,E)$ such that $f(u)\in N_F(u)$ for every $u\in L^p(I,E)$.\par Let $K\colon L^q(I,E)\to L^p(I,E)$ be given by \eqref{opK}. Observe that solving the equation $u=K(f(u))$ means to find a solution of the Hammerstein integral inclusion \eqref{hammerstein}.\par We have proven previously (see p.16) that the composite map $K\circ N_F$ satisfies the boundary condition \eqref{LS} provided $R>0$ is matched according to \eqref{ls2}. Therefore, operator $H\colon D_{L^p}(h,R)\to L^p(I,E)$, given by $H:=K\circ f$, will meet condition \eqref{LS} as well. This operator satisfies also condition \eqref{munch}. It can be shown in a manner strictly analogous to that demonstrated in the proof of Theorem \ref{hamth}. Since the space $E$ is separable, there is no need in particular to pass to the relative MNC $\beta_{E_0}$ and inequality \eqref{v_n} and successive ones hold under current assumptions as well. Finally, let us note that $H$ is admissible as a univalent continuous map. From the direct application of Theorem \ref{fixed} follows, in this regard, that it must possess a fixed point $u\in D_{L^p}(h,R)$.
\end{proof}

\noindent{\bf Example 7.}
The last example illustrates the application of Rothe-type fixed point argument, as a conclusion stemming from Theorem \ref{fixed}, in order to show the existence of solutions of the periodic problem \eqref{inclus} with non-convex perturbation term $F$.
\begin{theorem}
Let $(E,\mathbb{H},E^*)$ be an evolution triple such that $E\hookrightarrow \mathbb{H}$ compactly. Assume that the operator $A\colon I\times E\map E^*$ satisfies $(\A_1)$--$(\A_6)$, while the map $F\colon I\times {\mathbb H}\map{\mathbb H}$ fulfills conditions {\em (i)--(iv)} of Theorem \ref{nonconvex}. Suppose further that the inequality \eqref{stala} holds. Then problem \eqref{inclus} has at least one solution.
\end{theorem}

\begin{proof}
In accordane with what we have shown in the proof of Theorem \ref{nonconvex} there must exist a continuous selection $f\colon\h\to\h^*$ of the Nemytski\v{\i} operator $N_F$. The proof of Theorem \ref{th1} leaves no doubt that the external operator $K\colon\h^*\map\h$, given by $K:=(L+N_{A_h})^{-1}$, is an upper semicontinuous multimap with compact convex values. Therefore, the map $H\colon\h\map\h$ such that $H:=K\circ f$ is an admissible one. Proof of Theorem \ref{th1} also illustrates the hypothesis that operator $H\colon D_\h(0,R)\map\h$ satisfies the boundary condition \eqref{LS}, if you adopt that the radius $R$ is not less than $\f(||b||_q+||h||_{\e^*})/(d-c)\g)^\frac{q}{p}$. And finally, the proof of this theorem sets forth reasons for the compactness of the operator $H\colon D_\h(0,R)\map\h$. Hence, we may again evoke the fixed point result expressed in Theorem \ref{fixed} to show the existence of a fixed point of $H$ i.e., a solution of the problem \eqref{inclus}.
\end{proof}

\end{document}